\documentclass[twoside]{amsart}

\numberwithin{equation}{section}
\usepackage{amsmath,amssymb,amsfonts,amsthm,latexsym}
\usepackage{graphicx,color,enumerate}
\usepackage[margin=1in]{geometry}
\usepackage{longtable}

\newtheorem{theorem}{Theorem}[section]

\newtheorem{proposition}[theorem]{Proposition}
\newtheorem{corollary}[theorem]{Corollary}

\theoremstyle{definition}

\theoremstyle{remark}
\newtheorem{remark}[theorem]{Remark}

\newcommand{\C}{\mathbb{C}}
\newcommand{\Sp}{\mathbb{S}}

\newcommand{\R}{\mathbb{R}}
\newcommand{\T}{\mathbb{T}}

\newcommand{\Z}{\mathbb{Z}}

\newcommand{\SU}{\operatorname{SU}}
\newcommand{\SL}{\operatorname{SL}}

\newcommand{\co}{\colon\thinspace}

\newcommand{\bs}{\boldsymbol}

\newcommand{\Hilb}{\operatorname{Hilb}}
\newcommand{\on}{\mathit{on}}

\begin{document}

\title{Hilbert series associated to symplectic quotients by $\operatorname{SU}_2$}

\author[H.-C.~Herbig]{Hans-Christian Herbig}
\address{Departamento de Matem\'{a}tica Aplicada,
Av. Athos da Silveira Ramos 149, Centro de Tecnologia - Bloco C, CEP: 21941-909 - Rio de Janeiro, Brazil}
\email{herbighc@gmail.com}

\author[D.~Herden]{Daniel Herden}
\address{Department of Mathematics, Baylor University,
One Bear Place \#97328,
Waco, TX 76798-7328, USA}
\email{Daniel\_Herden@baylor.edu}

\author[C.~Seaton]{Christopher Seaton}
\address{Department of Mathematics and Computer Science,
Rhodes College, 2000 N. Parkway, Memphis, TN 38112}
\email{seatonc@rhodes.edu}

\thanks{C.S. was supported by the E.C.~Ellett Professorship in Mathematics;
H.-C.H. was supported by CNPq through the \emph{Plataforma Integrada Carlos Chagas.}}


\begin{abstract}
We compute the Hilbert series of the graded algebra of real regular functions on the symplectic quotient
associated to an $\SU_2$-module and give an explicit expression for the first nonzero coefficient of the
Laurent expansion of the Hilbert series at $t = 1$. Our expression for the Hilbert series indicates an
algorithm to compute it, and we give the output of this algorithm for representations of dimension at most $10$.
Along the way, we compute the Hilbert series of the module of covariants of an arbitrary
$\SL_2$- or $\SU_2$-module as well its first three Laurent coefficients.
\end{abstract}

\maketitle
\tableofcontents


\section{Introduction}
\label{sec:Intro}

Let $(M,\omega)$ be a smooth symplectic manifold with an action of a compact Lie group $G$ preserving the symplectic form $\omega$. Assume that the $G$-action admits a $G$-equivariant \emph{moment map} $J:M\to \mathfrak g^\ast$, where $\mathfrak g^\ast$ denotes the dual space of the Lie algebra $\mathfrak g$ of $G$. In other words, the infinitesimal
action of $\xi\in \mathfrak g$  is given by the Hamiltonian vector field $\{J_\xi, \:\}$ of the function $J_\xi$ obtained by pairing $J$ with $\xi$. If $0\in \mathfrak g^\ast$ is a regular value, the zero level $Z:=J^{-1}(0)$ is a closed submanifold of $M$ and the \emph{symplectic quotient} $M_0=Z/G$ (also known as Marsden-Weinstein quotient or Hamiltonian reduction) is a symplectic orbifold. Otherwise, as shown in \cite{SjamaarLerman}, $Z$ has locally conical singularities and $M_0=Z/G$ is a \emph{stratified symplectic space}. The strata $(M_0)_{(H)}$ of this stratification are indexed by conjugacy classes $(H)$ of isotropy subgroups $H\subseteq G$ and are given by the set of $G$-orbits of points whose isotropy group is in $(H)$. The components of $(M_0)_{(H)}$ are smooth symplectic  manifolds. The strata and their interrelationships can be recovered from
the Poisson algebra of smooth functions $\mathcal{C}^\infty(M_0)=\mathcal{C}^\infty(M)^G/\mathcal{I}_Z^G$, corresponding to the algebra of \emph{classical observables} of the system. Here $\mathcal{I}_Z^G$ denotes the $G$-invariant part of the ideal $\mathcal{I}_Z$ of smooth functions vanishing on $Z$.  It has been suggested to view $(M_0, \mathcal{C}^\infty(M_0))$ as a differential space in the sense of Sikorski (cf. \cite{Sniatycki, HuebschmannMemoirs, FarHerSea}). The authors adhere to this philosophy.

Symplectic quotients, or incarnations thereof, frequently arise as moduli spaces in gauge theory. This is due to the fact that the curvature of a connection can be interpreted as a moment map on the cotangent bundle of the space of connections (see e.g. \cite{AtiyahBott,DonaldsonKronheimer}). In many interesting situations, the moduli space is finite-dimensional but exhibits singularities. In this case, it can often be locally understood as a symplectic quotient in the sense defined above (see e.g. \cite{GHJW}). For example, the moduli space of flat $\operatorname{SU}_2$-connections on a Riemann surface has local models that are a symplectic reductions at zero angular momentum \cite{HuebschmannCleaningUp, CapeHerbigSeaton}. How to deal with the singularities is the subject of ongoing research, and the literature on the topic is extensive.

It is well known since \cite{ArmsGotayJennings,SjamaarLerman} that the symplectic slice theorem can be employed to study the singularities locally. It is therefore natural and convenient to focus our attention on the case when $(M,\omega)$ is  $(\mathbb C^n,\omega_0)$ with its standard
K\"ahler structure and the action of $G$ is a unitary representation. In this situation,
the moment map is given by homogeneous quadratic polynomials, see Equation~\eqref{def:quadraticmm}. Typically, invariant theory is used to describe the semialgebraic geometry of $M_0$ (see e.g. \cite{LMS,HuebschmannMemoirs,FarHerSea,CapeHerbigSeaton}). Closely related to $M_0$ is its real Zarisky closure
$\overline{M_0}^z$. It is the spectrum of the Noetherian $\mathbb R$-algebra
\[
    \mathbb R[M_0]=\mathbb R[\mathbb C^n]^G/(\mathcal{I}_Z\cap \mathbb R[\mathbb C^n]^G).
\]
This algebra comes with a canonical nonstandard $\mathbb Z^+$-grading and is a Poisson subalgebra of $\mathcal{C}^\infty (M_0)$. Versions of the complexification $\mathbb C[M_0]=\mathbb R[M_0]\otimes_\mathbb R\mathbb C$ have been intensely studied in geometric representation theory (see e.g. \cite{NakajimaDuke, GinzburgNakajima, CrawleyBoeveyNormality}). For a more careful discussion of the subtleties of the definition of the complex symplectic quotient and the relation to the \emph{2-large} property, see \cite{HerbigSchwarzSeaton2}.

A fundamental question related to symplectic quotients is the \emph{symplectomorphism problem}, i.e. classifying the singularities that arise in symplectic quotients up to symplectomorphism.
The question of when a symplectic quotient is (graded regularly) symplectomorphic to a finite unitary quotient (i.e. a symplectic quotient with $G$ a finite group) has been proposed in \cite{FarHerSea} and, to some extent, answered in \cite{HerbigSchwarzSeaton}. Considering the symplectomorphism problem the following basic questions arise:
\begin{enumerate}[(A)]
\item What are the general traits that all symplectic quotients have in common?
\item What are appropriate invariants suitable to distinguish them up to symplectomorphism?
\end{enumerate}
In progress towards answering both of these questions, the Hilbert series of the graded algebra of regular functions $\R[M_0]$ has been instrumental. For a general result addressing question
(A), we refer the reader to \cite{HerbigSchwarzSeaton2} where it has been shown that symplectic quotients have symplectic singularities. In \cite{HerbigSeaton} the Hilbert series of the algebra $\R[M_0]$ of symplectic circle quotients $M_0$ has been examined, and implications for the graded regular symplectomorphism problem have been discussed. We emphasize that the Hilbert series is, in general, not invariant under regular symplectomorphism. Its main virtues are its accessiblility and that it encodes valuable information about the variety $\overline{M_0}^z$ (e.g. dimension, $a$-invariant, Gorensteinness). We mention that in recent years, Hilbert series calculations have gained considerable attention by people working in supersymmetric Yang-Mills theory as part of the \emph{plethystic program} (see e.g. \cite{Hanany}).

In this paper, we study the Hilbert series of the algebra of regular functions
$\R[M_0]$ of the symplectic quotient of a (finite-dimensional) unitary representation $V$ of the group $G=\SU_2$. Let $r>1$ be an integer and let $\Gamma:=(\mathbb Z^+)^r$
be the semigroup of multi-indices $\bs n=(n_1,n_2,\dots,n_r)$.
Recall that the \emph{Hilbert series}
of a $\Gamma$-graded vector space $X=\bigoplus_{\bs n\in \Gamma} X_{\bs n}$ with $X_{\bs n}$ finite-dimensional for all $\bs n\in \Gamma$ is the generating function
\begin{align*}
    \Hilb_X (\bs{t})    &   =\Hilb_X (t_1,t_2,\dots,t_r):=\sum_{\bs n\in \Gamma} \dim( X_{\bs n})\: \bs t^{\bs n}
        \\&=    \sum_{n_1,n_2,\dots,n_r=0}^\infty \dim( X_{(n_1,n_2,\dots,n_r)})\:
            t_1^{n_1}t_2^{n_2}\cdots t_r^{n_r}\in\mathbb Q[\![\bs t]\!]=\mathbb Q[\![ t_1,t_2,\dots,t_r]\!].
\end{align*}
If $r=1$, we refer to $\Hilb_X (\bs{t})$ as a \emph{univariate Hilbert series} and set $t:=t_1$; otherwise we say it is \emph{multivariate}. A univariate Hilbert series corresponding to grading by total degree can be determined from a multivariate Hilbert series $\Hilb_X (\bs t)$ by specialization $\Hilb_X (t)=\Hilb_X (t,t,\dots,t)$.

The Hilbert series of primary interest for us is the univariate Hilbert series
\[
    \Hilb_{(G,V)}^{on}(t):=\Hilb_{\R[M_0]}(t)=\sum_{n=0}^\infty \dim(\R[M_0]_n)\: t^n.
\]
We refer to it as the \emph{on-shell Hilbert series}. If the group $G$ is understood we will write simply $\Hilb_{V}^{on}(t)$. It is known \cite{HerbigSchwarzSeaton} that $\R[M_0]$ is a graded Gorenstein algebra if the complexified representation $(G_\C,V)$ is 2-large (this assumption is not necessary if $G=\SU_2$ or the a torus).
It then follows that $\Hilb_{(G,V)}^{on}(t)$ is rational and satisfies the functional equation $\Hilb_{(G,V)}^{on}(t^{-1})=(-t)^d\Hilb_{(G,V)}^{on}(t)$ where $d$ is the Krull dimension of $\overline{M_0}^z$.
In order to understand $\Hilb_{(\SU_2,V)}^{on}(t)$, we employ the univariate Hilbert series $\Hilb_{\C[V\oplus V^\ast]^{\SL_2}}(t)$ of the polynomial $\SL_2 = (\SU_2)_\C$-invariants of the
cotangent lifted representation $V\oplus V^\ast$.

Here we can use the previous work \cite{CayresPintoHerbigHerdenSeaton} on the univariate Hilbert series
\[
    \Hilb_{\C[W]^{\SL_2}}(t)=\sum_{n=1}^\infty \dim(\C[W]^{\SL_2}_n) \:t^n
\]
of polynomial invariants of a finite-dimensional $\SL_2$-representation $W$. In the Laurent expansion
\[
    \Hilb_{\C[W]^{\SL_2}}(t)=\sum_{n=0}^\infty \gamma_n(W) (1-t)^{n-D+3},
\]
the pole order equals $D-3$ (where $D=\dim_\C W$) if $(\SL_2,W)$ if $W$ is $1$-large. Moreover, formulas in terms of Schur functions for the coefficients $\gamma_0,\gamma_1$ and $\gamma_2$ have been elaborated and an algorithm computing $\Hilb_{\C[W]^{\SL_2}}(t)$ has been presented and implemented.
Many ideas in \cite{CayresPintoHerbigHerdenSeaton} are adaptations of developments in \cite{HerbigSeaton}, where $\Hilb_{(G,V)}^{on}(t)$ in the case of $G=\Sp^1$ has been examined. Similar considerations have also been used in \cite{CowieHerbigSeatonHerden} to examine $\Hilb_{\C[W]^{\C^{\times}}}(t)$ for a linear representation $W$ of the complex circle $\C^{\times}$.

If $G$ is an $\ell$-dimensional torus $\mathbb T^\ell=(\Sp^1)^\ell$, the complexification
$G_{\C}$ is $\mathbb T^\ell_{\C}=(\C^\times)^\ell$. It can be assumed without loss of generality that the representation $(\mathbb T^\ell_{\C} ,V)$ is stable and faithful. With this assumption, we have the simple relationship
\[
    \Hilb_{(G,\mathbb T^\ell)}^{on}(t)=(1-t^2)^\ell\Hilb_{\C[V\oplus V^\ast]^{\mathbb T^\ell_{\C}}}(t).
\]
If the group $G$ is nonabelian, the relationship is more complicated. We instead systematize the method of \cite[Subsection 6.3]{HerbigSchwarzSeaton2}, i.e. we take the $\SU_2$-invariant part of the Koszul complex on the moment map.  Recall that under mild assumptions (more precisely if $(G_{\C},V)$ is 1-large), the Koszul complex is a resolution of the algebra of on-shell function \cite{HerbigSchwarz}.
This enables us to express $\Hilb_{(\SU_2,V)}^{on}(t)$ as a linear combination of $\Hilb_{\C[V\oplus V^\ast]^{\SL_2}}(t)$ and $\Hilb_{(\C[V\oplus V^\ast]\otimes V_2)^{\SL_2}}(t)$ (see Proposition~\ref{prop:OnShellHilb}). Here $V_2$ is the irreducible $\SL_2$-representation given by binary forms of degree 2 (or in physics terminology, spin 1). The space $(\C[V\oplus V^\ast]\otimes V_2)^{\SL_2}$ can be interpreted as the module of $V_2$-covariants, i.e. the space of ${\SL_2}$-equivariant linear maps from $\C[V\oplus V^\ast]$ to $V_2$.

The missing piece is the Hilbert series $\Hilb_{(\C[W]\otimes V_L)^{\SL_2}}(t)$ of the module of covariants, and
Section~\ref{sec:Covariants} is dedicated to elaborating formulas for it. We take the liberty of treating instead of binary forms $V_2$ of degree two the more general case of  binary forms $V_L$ of degree $L$. If $W$ can decomposed into irreducibles $W=V_{d_1}\oplus V_{d_2}\oplus \cdots \oplus V_{d_r}$, the space $(\C[W]\otimes V_L)^{\SL_2}$
carries a natural $(\mathbb Z^+)^r$-grading (see for example \cite{BroerCM}). This makes it possible to refine the analysis and look into the multivariate Hilbert series of the covariants
\begin{equation}
\label{def:multivarateHilbert}
    \Hilb_{(\C[W]\otimes V_L)^{\SL_2}}(t_1,t_2,\dots,t_r)
        =   \sum_{n_1,n_2,\dots,n_r=0}^\infty \dim_{\C}(\C[W]\otimes V_L)^{\SL_2}_{(n_1,n_2,\dots,n_3)}
            \: t_1^{n_1}t_2^{n_2}\dots t_r^{n_r}.
\end{equation}
The first preparatory result is Proposition~\ref{prop:CovarMultivarHilbSer}
that presents the multivariate Hilbert series in terms if sums over roots of unity and analytic continuation. Similar formulas have been found by Brion \cite{Brion} for the algebra of $\SL_2$-invariant polynomial functions. For algorithms concerning the multivariate and unvariate Hilbert series and other approaches to the computation of the Hilbert series
of covariants in some cases, see
\cite{BedratyukSL2MultigradPoincare,BedratyukPoincareCovariants,BedratyukBivarPoincare,BedratyukIlashCovariants}.
In addition, Broer has considered the computation of the Hilbert series for $\SL_2$-covariants, see
\cite{BroerNewMethod,BroerCovariants}.

In Subsection~\ref{subsec:CovarHilbSerUnivar} we specialize to the univariate Hilbert series  $\Hilb_{(\C[W]\otimes V_L)^{\SL_2}}(t)$ and deduce a formula for it in Theorem~\ref{thrm:CovarUnivarHilbSer}. As a first consequence,
we describe an algorithm that we have implemented to compute the covariants.
Along with Proposition~\ref{prop:OnShellHilb}, this algorithm enables the computation of $\Hilb_{(\SU_2,V)}^{on}(t)$,
and we have produced a list of examples of in Table~\ref{tab:HilbM0LowDim} of Appendix~\ref{app:HSerData}.
In Subsection~\ref{subsec:CovarHilbSerLaurent} we use Theorem~\ref{thrm:CovarUnivarHilbSer} to deduce formulas
for the first three coefficients $\gamma_{0,L}$, $\gamma_{1,L}$ and $\gamma_{2,L}$ in the Laurent expansion of $\Hilb_{(\C[W]\otimes V_L)^{\SL_2}}(t)$, see Theorems~\ref{thrm:CovarGamma0}, \ref{thrm:CovarGamma1} and \ref{thrm:CovarGamma2}.

Finally, in Section~\ref{sec:SympQuots} we are able to derive our main result, Theorem~\ref{thrm:OnShellHilbGam0}, that
expresses the leading coefficient  $\gamma_0^{on}(V)$ of the Laurent expansion
of $\Hilb_{(\SU_2,V)}^{on}(t)$ in terms of the weights of the representation using Schur functions. The main difficulty in achieving this result is that there occur certain cancelations in the leading terms of the expansion;
see Equation~\eqref{eq:OnShellExpanN1N2} and the discussion that follows.
For this reason, we are forced to calculate the first three terms in the Laurent expansions of $\Hilb_{(\C[W]\otimes V_2)^{\SL_2}}(t)$ and $\Hilb_{\C[V\oplus V^\ast]^{\SL_2}}(t)$. The latter calculation was done in the previous work \cite{CayresPintoHerbigHerdenSeaton}.
From the graded Gorensteinness of $\R[M_0]$ and the Gorensteinnes of $\C[V\oplus V^\ast]^{\SL_2}$ we deduce further relations among the Laurent coefficients. For the effects of Gorensteinness on the Laurent expansion of the Hilbert series see our previous article \cite{HerbigHerdenSeaton2}.

We draw some empirical conclusions and suggest that $\gamma_0^{on}(V)$ can be seen as a measure of the degree of reducibility of the representation $V$.
In Appendix~\ref{app:HSerData}, we list $\Hilb_{(\SU_2,V)}^{on}(t)$, $\gamma_0^{on}$ and $\gamma_2^{on}$ for $\dim_{\R} M_0$ between $2$ and $14$ (note that graded Gorensteinness implies $\gamma_1^{\on} = 0$ by \cite{HerbigHerdenSeaton}).
Data for higher dimensional cases are available from the authors by request. In Appendix~\ref{app:Gam0Plot} we provide
visualizations of the behavior of $\gamma_0^{on}$ for $\dim_{\R} M_0$ between $2$ and $38$.

\section*{Acknowledgements}

We would like to thank Gerald Schwarz for teaching us the method used to compute $\Hilb_V^{\on}(t)$
via the Koszul complex, which was first applied in \cite[Section 6.3]{HerbigSeaton2}.
Seaton thanks the E.C.~Ellett Professorship in Mathematics and Herbig thanks CNPq for financial support.


\section{Background}
\label{sec:Background}

Let $G$ be a compact Lie group and $V$ a unitary $G$-module. We briefly recall the
definition of the corresponding symplectic quotient $M_0$ and its graded Poisson
algebra of regular functions $\R[M_0]$. The reader is referred to
\cite[Sections 2.1 and 4]{FarHerSea} or \cite[Section 2.2]{HerbigSchwarzSeaton2}
for more details.

The moment map associated to $(G,V)$ is the regular quadratic map
\begin{align}\label{def:quadraticmm}
    J\co V\to \mathfrak{g}^\ast,
    \quad\quad
    v\mapsto (J(v), \xi)
        :=   \frac{\sqrt{-1}}{2} \langle v, \xi.v \rangle,
\end{align}
where $(J(v), \xi)$ denotes the dual pairing of $J(v)\in\mathfrak{g}^\ast$
with $\xi\in\mathfrak{g}$, $\langle\cdot,\cdot\rangle$ denotes the Hermitian
scalar product, and $\xi.v$ denotes the infinitesimal action. We let $Z$
denote the $G$-invariant set $J^{-1}(0)\subset V$, informally called the
\emph{shell}, and let $M_0 := Z/G$ denote the symplectic quotient. When $G$ is not
discrete, $0\in\mathfrak{g}^\ast$ is not a regular value of $J$ so that $Z$ is
not a manifold. Instead, $M_0$ is stratified into symplectic manifolds and hence
has the structure of a \emph{stratified symplectic space} by \cite{SjamaarLerman}.
Letting $\mathcal{C}^\infty(V)$ denote the
algebra of real-valued smooth functions on $V$, $\mathcal{I}_Z$ denote the ideal of
$\mathcal{C}^\infty(V)$ consisting of those functions that vanish on $Z$, and
$\mathcal{I}_Z^G := \mathcal{C}^\infty(V)^G\cap\mathcal{I}_Z$ the invariant part of
$\mathcal{I}_Z$, the \emph{smooth structure} of $M_0$ is defined to be the quotient
\[
    \mathcal{C}^\infty(M_0) := \mathcal{C}^\infty(V)^G/\mathcal{I}_Z^G.
\]
The algebra $\mathcal{C}^\infty(M_0)$ inherits a Poisson structure from the usual
Poisson bracket on $\mathcal{C}^\infty(V)$, and $M_0$ along with this Poisson algebra
has the structure of a \emph{Poisson differential space}, see \cite[Section 4.1]{FarHerSea}.

The \emph{Poisson algebra of regular functions on $M_0$}, $\R[M_0] = \R[V]^G/(\mathcal I_Z\cap \R[V]^G)$, is a Poisson subalgebra of $\mathcal{C}^\infty(M_0)$. It is a $\Z^+$-graded Poisson algebra with Poisson bracket being of degree $-2$.
In the language of \cite[Section 4.2]{FarHerSea}, $(\R[M_0],\{\:,\:\})$ together with the fundamental
polynomial invariants provide a \emph{global chart} for the Poisson differential space $M_0$.
We let $\Hilb_{\R[V]^G}(t)$ denote the Hilbert series of the graded algebra of real
$G$-invariant polynomials and $\Hilb_{(G,V)}^{\on}(t)$ (or simply $\Hilb_{V}^{\on}(t)$
when $G$ is clear from the context) denote the Hilbert series of the graded algebra
$\R[M_0]$ of ``on-shell" invariants, i.e. real regular functions on $M_0$.

Although our primary interest is the real symplectic quotient $M_0$ by the compact Lie group
$G$ and hence the real polynomial invariants, it is often helpful to consider the
corresponding complexifications. The action of $G$ on $V$ extends to an action of
the complexification $G_\C$ on $V$. If $(V,G_\C)$ is \emph{$1$-large}
(see \cite{GWSlifting}, \cite{HerbigSchwarz}, or \cite{HerbigSchwarzSeaton2} for the
definition), then by \cite[Corollary 4.3]{HerbigSchwarz}, the quadratic components of
$J$ generate a homogeneous real ideal $(J)$ of $\R[V]$ so that $\R[M_0] = \R[V]^G/(J)^G$.
Complexifying the underlying real vector space of $V$
yields $V\otimes_\R\C$, which is isomorphic to $V \oplus V^\ast$ as a $G$-module
(or equivalently a $G_\C$-module); we refer to $V \oplus V^\ast$ as the
\emph{cotangent lift} of $V$. We let
$\mu = J\otimes_\R\C\co V\oplus V^\ast \to \mathfrak{g}_\C^\ast$ denote the
complexification of the moment map, where $\mathfrak{g}_\C^\ast$ is the Lie
algebra of $G_\C$, and refer to the subscheme of $V \oplus V^\ast$ associated to
the ideal $(\mu)$ as the \emph{complex shell}. The associated \emph{complex
symplectic quotient} is defined to be
$\operatorname{Spec}\big(\R[M_0]\otimes_\R \C\big)$, which in the case of
$1$-large $(V, G_\C)$ coincides with the spectrum of
$\C[V\oplus V^\ast]^{G_\C}/(\mu)^{G_\C}$.
In particular, note that the Hilbert series of the graded algebras $\R[M_0]$ and
$\C[V\oplus V^\ast]^{G_\C}/(\mu)^{G_\C}$ coincide in the $1$-large case.
In non-$1$-large cases, this definition of the complex quotient is not standard;
see \cite[Section 2.2]{HerbigSchwarzSeaton2} for a thorough discussion.

For the rest of the paper, we specialize to the case where $G = \SU_2$ so that $G_\C = \SL_2$.
Let $V$ be a finite-dimensional unitary $\SU_2$-module, or equivalently a
finite-dimensional $\SL_2$-module, and assume for simplicity that $V^{\SU_2} = \{0\}$.
We use $V_d$ to denote the unique irreducible $\SU_2$-module of
dimension $d+1$ on binary forms of degree $d$. We assume that $(V,\SL_2)$ is $1$-large,
which by \cite[Theorem 3.4]{HerbigSchwarz} is true for all $V$ not isomorphic to
$V_1$, $2V_1$, nor $V_2$. Then by \cite[Lemma 2.1(2)]{HerbigSchwarz},
the components of $\mu$ form a regular sequence in $\C[V\oplus V^\ast]$ so that
the Koszul complex of $\mu$ is a free resolution of $\C[V\oplus V^\ast]/(\mu)$,
see \cite[Corollary 1.6.14(b)]{BrunsHerzog}. Let $S:=\C[V\oplus V^\ast]$, and then
we have an exact sequence
\[
    0
    \longrightarrow     S \simeq S\otimes \wedge^3 \mathfrak{sl}_2
    \longrightarrow     S\otimes \wedge^2 \mathfrak{sl}_2
    \longrightarrow     S\otimes \mathfrak{sl}_2
    \longrightarrow     S
    \longrightarrow     S/\mathfrak{sl}_2
    \longrightarrow     0,
\]
where the elements of $\mathfrak{sl}_2$ are in degree $2$,
the elements of $\wedge^2\mathfrak{sl}_2$ are in degree $4$, and
the elements of $\wedge^3 \mathfrak{sl}_2$ are in degree $6$.
As $\mathfrak{sl}_2$ and $\wedge^2\mathfrak{sl}_2$ are isomorphic as $\SL_2$-
(or $\SU_2$-)modules, this exact sequence can be rewritten as
\[
    0
    \longrightarrow     S
    \longrightarrow     S\otimes \mathfrak{sl}_2
    \longrightarrow     S\otimes \mathfrak{sl}_2
    \longrightarrow     S
    \longrightarrow     S/\mathfrak{sl}_2
    \longrightarrow     0.
\]
Taking invariants, we have
\[
    0
    \longrightarrow     S^{\SL_2}
    \longrightarrow     (S\otimes \mathfrak{sl}_2)^{\SL_2}
    \longrightarrow     (S\otimes \mathfrak{sl}_2)^{\SL_2}
    \longrightarrow     S^{\SL_2}
    \longrightarrow     (S/\mathfrak{sl}_2)^{\SL_2}
    \longrightarrow     0.
\]
Now, $S^{\SL_2} = \C[V\oplus V^\ast]^{\SL_2} = \C[V\oplus V^\ast]^{\SU_2}$,
and as $\mathfrak{sl}_2\simeq V_2$ as $\SL_2$- or $\SU_2$-modules,
$(S\otimes \mathfrak{sl}_2)^{\SL_2}$ is isomorphic to the module of covariants
$(\C[V\oplus V^\ast]\otimes V_2)^{\SL_2} = (\C[V\oplus V^\ast]\otimes V_2)^{\SU_2}$.
Finally, $(S/\mathfrak{sl}_2)^{\SL_2} = \R[M_0]\otimes_\R \C$. Therefore, letting
$\Hilb_{(\C[V\oplus V^\ast]\otimes V_2)^{\SL_2}}(t)$ denote the Hilbert series of the
module of covariants, we have the following.

\begin{proposition}
\label{prop:OnShellHilb}
Let $V$ be a unitary $\SU_2$-representation with $V^{\SU_2} = \{0\}$ and assume that $V$ is not isomorphic
to $V_1$, $2V_1$ nor $V_1\oplus V_2$. Then the on-shell Hilbert series of the graded algebra of regular functions
on the symplectic quotient $M_0$ is given by
\begin{equation}
\label{eq:OnShellHilb}
    \Hilb_V^{\on}(t) =
    (1 - t^6) \Hilb_{\C[V\oplus V^\ast]^{\SL_2}}(t) + (t^4 - t^2) \Hilb_{(\C[V\oplus V^\ast]\otimes V_2)^{\SL_2}}(t).
\end{equation}
\end{proposition}

\begin{remark}
\label{rem:Non1Large}
Note that in each of the non-$1$-large cases, the symplectic quotients are isomorphic to
quotients by finite groups, see \cite[Section 5]{HerbigSchwarzSeaton}
or \cite{ArmsGotayJennings,GotayBos}, and hence easily handled individually.
Specifically, the symplectic quotient associated to $V_1$ is a point, while the
(real) symplectic quotients associated to $2V_1$ and $V_2$ are both isomorphic to
$\C/\pm 1$.
\end{remark}

In Section~\ref{sec:Covariants}, we will compute the Hilbert series
$\Hilb_{(\C[V\oplus V^\ast]\otimes W)^{\SL_2}}(t)$ of covariants in general as well as the
first few Laurent coefficients of this series; note that the case of invariants, i.e. $W = \C$,
was treated in \cite{CayresPintoHerbigHerdenSeaton}. In Section~\ref{sec:SympQuots}, we will
apply these results to a computation of the first nonzero Laurent coefficient of
$\Hilb_V^{\on}(t)$.


\section{The Hilbert Series of the Module of Covariants}
\label{sec:Covariants}

Let $W$ and $U$ be finite-dimensional unitary $\SU_2$-modules, equivalently finite-dimensional
$\SL_2$-modules. In this section, we compute the Hilbert series $\Hilb_{(\C[W]\otimes U)^{\SL_2}}(t)$
of the module of covariants $W \to U$ as well as the first three Laurent coefficients of
$\Hilb_{(\C[W]\otimes U)^{\SL_2}}(t)$ at $t=1$.

We use the following notation from \cite{CayresPintoHerbigHerdenSeaton}.
Let $W = \bigoplus_{k=1}^r V_{d_k}$ indicate the decomposition of $W$ into irreducible representations,
where $V_{d}$ denotes the unique irreducible representation of $\SL_2$ of dimension $d+1$, given by
binary forms of degree $d$. Let $D = \dim_{\C} W = r + \sum_{k=1}^r d_k$. Define
\[
    \Theta  :=      \{(k,i)\in \Z\times\Z : 1 \leq k \leq r, \; 0 \leq  i \leq d_k \},
\]
and for each $(k,i)\in\Theta$, let $a_{k,i} := 2i - d_k$. Set
$\bs{a}_{\Theta} := \big(a_{k,i} : (k,i)\in\Theta\big)$, and then the coordinates of $\bs{a}_{\Theta}$
are the weights of the representation $W$. Similarly, let
\[
    \Lambda :=      \{(k,i)\in \Z\times\Z : 1 \leq k \leq r, \; \lfloor d_k/2 \rfloor + 1 \leq  i \leq d_k \},
\]
let $C:=|\Lambda|$, and let $\bs{a} := \big(a_{k,i} : (k,i)\in\Lambda\big)$; then $\bs{a}$ lists the
positive weights. When the representation $W$ is not clear from the context, we will use
the notation $\Theta_W$, $\Lambda_W$, $\bs{a}_W$, etc.

The Hilbert series $\Hilb_{(\C[W]\otimes U)^{\SL_2}}(t)$ can be expressed as an integral over the Cartan torus
$\T$ of $\SL_2$ using the Molien-Weyl formula; see \cite{BroerCovariants,BroerNewMethod}. That is,
\[
    \Hilb_{(\C[W]\otimes U)^{\SL_2}}(t)
    =   \int\limits_{z\in\T} \frac{\chi_{U}(z) \, d\mu(z)}
        {\det_{W^\ast}(1 - z t)},
        \quad\quad\quad |t| < 1,
\]
where $\chi_U$ is the character associated to $U$ and
$\mu(z)$ is a Haar measure on $\SU_2$ such that $\int_{\T} d\mu = 1$. By a minor modification
of the proof to the Molien-Weyl formula, the multivariate Hilbert series can be expressed as
\[
    \Hilb_{(\C[W]\otimes U)^{\SL_2}}(t_1,\ldots,t_r)
    =   \int\limits_{z\in\T} \frac{\chi_{U}(z) \, d\mu(z)}
        {\prod\limits_{k=1}^r \det_{V_k^\ast}(1 - z t_k)},
        \quad\quad\quad |t_k| < 1.
\]
This has been used in \cite[Equation (13)]{StanleyInvarFinGrp} for finite groups
and \cite[Section IV]{Forger} for the decomposition of a real representation into holomorphic
and antiholomorphic parts. It is an immediate consequence that if $U = U_1\oplus U_2$, we have
\[
    \Hilb_{(\C[W]\otimes U)^{\SL_2}}(t)
        =   \Hilb_{(\C[W]\otimes U_1)^{\SL_2}}(t) + \Hilb_{(\C[W]\otimes U_2)^{\SL_2}}(t),
\]
and the same holds for the multivariate Hilbert series $\Hilb_{(\C[W]\otimes U)^{\SL_2}}(t_1,\ldots,t_r)$;
for this reason, we restrict our attention to the case $U = V_L$ is irreducible with no loss of generality.

Identifying $\T$ with the circle $\Sp^1 \subset \C$, we have
\[
    d\mu = \frac{(1 - z^2)\,dz}{2\pi \sqrt{-1} }.
\]
Then the character $\chi_{V_L}(z)$ is given by
\[
    \chi_{V_L}(z) = z^L + z^{L-2} + \cdots + z^{2-L} + z^{-L},
\]
so that, noting that $W = W^\ast$ in this case, we have
\[
    \int\limits_{z\in\Sp^1} \frac{\chi_{V_L}(z) \, d\mu(z)}
        {\det_W(1 - z t)}
    =
    \frac{1}{2\pi\sqrt{-1}}
    \int\limits_{z\in\Sp^1} \frac{ z^{-L} - z^{L+2} }
        {z \det_W(1 - z t)}\,dz.
\]
We define
\[
    \Upsilon_{W,\ell}(t)    :=
    \frac{1}{2\pi\sqrt{-1}}
    \int\limits_{z\in\Sp^1} \frac{ z^\ell dz}
        {z\det_W(1 - z t)}
\]
and then $\Hilb_{(\C[W]\otimes V_L)^{\SL_2}}(t) = \Upsilon_{W,-L}(t) - \Upsilon_{W, L+2}(t)$;
similarly,
\[
    \Upsilon_{W,\ell}(t_1,\ldots,t_r)     :=
    \frac{1}{2\pi\sqrt{-1}}
    \int\limits_{z\in\Sp^1} \frac{ z^\ell dz}
        {z\prod\limits_{k=1} \det_{V_k^\ast}(1 - z t_k)},
\]
and $\Hilb_{(\C[W]\otimes V_L)^{\SL_2}}(t_1,\ldots,t_r) = \Upsilon_{W,-L}(t_1,\ldots,t_r) - \Upsilon_{W,L+2}(t_1,\ldots,t_r)$.
Using the notation established above, we can express
\begin{equation}
\label{eq:GenIntegral}
    \Upsilon_{W,\ell}(t)
    =
    \frac{1}{2\pi\sqrt{-1}}\int\limits_{\Sp^1}
        \frac{ z^{\ell-1} \, dz}
        {\prod\limits_{d_k\in 2\Z}(1 - t) \prod\limits_{(k,i)\in\Lambda}
            (1 - t z^{-a_{k,i}})(1 - t z^{a_{k,i}})}
\end{equation}
and
\begin{equation}
\label{eq:GenIntegralMultivar}
    \Upsilon_{W,\ell}(t_1,\ldots,t_r)
    =
    \frac{1}{2\pi\sqrt{-1}}\int\limits_{\Sp^1}
        \frac{ z^{\ell-1} \, dz}
        {\prod\limits_{d_k\in 2\Z}(1 - t_k) \prod\limits_{(k,i)\in\Lambda}
            (1 - t_k z^{-a_{k,i}})(1 - t_k z^{a_{k,i}})}.
\end{equation}

To compute the Hilbert series, we will evaluate the integrals in Equations~\eqref{eq:GenIntegral}
and \eqref{eq:GenIntegralMultivar} using the methods of \cite[Section~4.6.1 and 4.6.4]{DerskenKemperBook}.
This was accomplished in \cite{CayresPintoHerbigHerdenSeaton} for invariants, i.e. the case $V_L = \C$
is trivial. For arbitrary $V_L$, we define the quantity
\begin{equation}
\label{eq:DefNu}
    \nu_{W,L}    :=  L + 1 - \sum_{(k,i)\in\Lambda} a_{k,i}.
\end{equation}
When $\nu_{W,L} \leq 0$, the computation of $\Hilb_{(\C[W]\otimes V_L)^{\SL_2}}(t)$ follows that
of $\Hilb_{\C[W]^{\SL_2}}(t)$ in \cite{CayresPintoHerbigHerdenSeaton} with little change, while
a new residue appears when $\nu_{W,L} > 0$.

We note that the Hilbert series $\Hilb_{(\C[W]\otimes U)^{\SL_2}}(t_1,\ldots,t_r)$ and
$\Hilb_{(\C[W]\otimes U)^{\SL_2}}(t)$ could also be computed using the Clebsch-Gordan decomposition
of $\C[W]\otimes U$ and the results of \cite{CayresPintoHerbigHerdenSeaton}.
We found the approach outlined and implemented here to be the most direct.


\subsection{The Multivariate Hilbert Series}
\label{subsec:CovarHilbSerMultivar}

Our first goal is the computation of the $\Gamma$-graded Hilbert Series
$\Hilb_{(\C[W]\otimes V_L)^{\SL_2}}(t_1,\ldots,t_r)$
of the covariants $W \to V_L$ by evaluating the integral in Equation~\eqref{eq:GenIntegralMultivar}.
To understand the integrand at $z = 0$, we rewrite it as
\[
    \frac{ z^{\ell-1 + \sum_{(k,i)\in\Lambda} a_{k,i}}}
        {\prod\limits_{d_k\in 2\Z}(1 - t_k) \prod\limits_{(k,i)\in\Lambda}
            (z^{a_{k,i}} - t_k)(1 - t_k z^{a_{k,i}})}.
\]
Recalling that $a_{k,i} > 0$ for each $(k,i)\in\Lambda$, it is clear that if $\ell = L + 2$, then
this expression is holomorphic at $z = 0$. Similarly, if $\ell = -L$, then the numerator is equal to
\[
    z^{-L-1 + \sum_{(k,i)\in\Lambda} a_{k,i}}
    =
    z^{- \nu_{W,L}},
\]
so the integrand is holomorphic at $z = 0$ if and only if $\nu_{W,L} \leq 0$.

Assume each $t_k$ is fixed with $|t_k| < 1$. As each $a_{k,i} > 0$, the poles inside the unit disk
away from $z=0$  occur when the factors $1 - t_k z^{-a_{k,i}}$ vanish. Hence, the poles are at
points of the form $z = \zeta t_k^{1/a_{k,i}}$ where $\zeta$ is an $a_{k,i}$th root of unity
and $t_k^{1/a_{k,i}}$ is defined using a fixed branch of the logarithm whose domain includes
each of the finitely many points $t_k$ to which it is applied.

Fix a $(K,I)\in\Lambda$ and an
$a_{K,I}$th root of unity $\zeta_0$, and then we express
\begin{align*}
    &\frac{ z^{\ell-1} }
    {\prod\limits_{d_k\in 2\Z}(1 - t_k) \prod\limits_{(k,i)\in\Lambda}
        (1 - t_k z^{-a_{k,i}})(1 - t_k z^{a_{k,i}})}
    \\&\quad\quad=
    \frac{ z^{a_{K,I} + \ell - 1} }
    {(z^{a_{K,I}} - t_K) (1 - t_K z^{a_{K,I}}) \prod\limits_{d_k\in 2\Z}(1 - t_k)
        \prod\limits_{\substack{(k,i)\in\Lambda\smallsetminus \\ \{(K,I)\}}}
            (1 - t_k z^{-a_{k,i}})(1 - t_k z^{a_{k,i}}) }
    \\&\quad\quad=
    \frac{ z^{a_{K,I} + \ell - 1}}
    {(1 - t_K z^{a_{K,I}})
        (z - \zeta_0 t_K^{1/a_{K,I}})
        \prod\limits_{\substack{\zeta^{a_{K,I}} = 1\\ \zeta\neq\zeta_0}} (z - \zeta t_K^{1/a_{K,I}})
        \prod\limits_{d_k\in 2\Z}(1 - t_k) \prod\limits_{\substack{(k,i)\in\Lambda\smallsetminus \\ \{(K,I)\}}}
            (1 - t_k z^{-a_{k,i}})(1 - t_k z^{a_{k,i}}) }.
\end{align*}
Hence, we have a simple pole at $z = \tau := \zeta_0 t_K^{1/a_{K,I}}$, and the residue at $z = \tau$
is given by
\begin{align*}
    &\frac{ \tau^{a_{K,I}+\ell - 1}}
    {(1 - t_K \tau^{a_{K,I}})
        \prod\limits_{\substack{\zeta^{a_{K,I}} = 1\\ \zeta\neq\zeta_0}} (\tau - \zeta t_K^{1/a_{K,I}})
        \prod\limits_{d_k\in 2\Z}(1 - t_k) \prod\limits_{\substack{(k,i)\in\Lambda\smallsetminus \\ \{(K,I)\}}}
            (1 - t_k \tau^{-a_{k,i}})(1 - t_k \tau^{a_{k,i}}) }
    \\&=
    \frac{ \tau^{a_{K,I} + \ell - 1} }
    {(1 - t_K^2) \tau^{a_{K,I} - 1}
        \prod\limits_{\substack{\zeta^{a_{K,I}} = 1\\ \zeta\neq 1}} (1 - \zeta)
        \prod\limits_{d_k\in 2\Z}(1 - t_k) \prod\limits_{\substack{(k,i)\in\Lambda\smallsetminus \\ \{(K,I)\}}}
            (1 - t_k \tau^{-a_{k,i}})(1 - t_k \tau^{a_{k,i}}) }
    \\&=
    \frac{ \zeta_0^\ell t_K^{\ell/a_{K,I}} }
    {a_{K,I}(1 - t_K^2) \prod\limits_{d_k\in 2\Z}(1 - t_k)
        \prod\limits_{\substack{(k,i)\in\Lambda\smallsetminus \\ \{(K,I)\}}}
            (1 - \zeta_0^{-a_{k,i}} t_k t_K^{-a_{k,i}/a_{K,I}})
            (1 - \zeta_0^{a_{k,i}} t_k t_K^{a_{k,i}/a_{K,I}})
            }.
\end{align*}
Summing over each choice of $(K,I)$ and $\zeta_0$, we have that when $\nu_{W,L} \leq 0$.,
\[
    \Upsilon_{W,\ell}(t_1,\ldots,t_k)
    =
    \sum\limits_{(K,I)\in\Lambda} \sum\limits_{\zeta^{a_{K,I}} = 1}
    \frac{ \zeta^{\ell} t_K^{\ell/a_{K,I}} }
    {a_{K,I}(1 - t_K^2) \prod\limits_{d_k\in 2\Z}(1 - t_k)
        \prod\limits_{\substack{(k,i)\in\Lambda\smallsetminus \\ \{(K,I)\}}}
            (1 - \zeta^{-a_{k,i}} t_k t_K^{-a_{k,i}/a_{K,I}})
            (1 - \zeta^{a_{k,i}} t_k t_K^{a_{k,i}/a_{K,I}})
            }.
\]

When $\nu_{W,\ell} > 0$, we have in addition a pole at $z = 0$ in the integrand of
$\Upsilon_{W, -L}(t_1,\ldots,t_r)$,
\[
    \frac{1}
        {z^{\nu_{W,L}} \prod\limits_{d_k\in 2\Z}(1 - t_k) \prod\limits_{(k,i)\in\Lambda}
            (z^{a_{k,i}} - t_k)(1 - t_k z^{a_{k,i}})}.
\]
The residue at $z = 0$ is given by
\[
    \frac{1}{(\nu_{W,\ell} - 1)! \prod\limits_{d_k\in 2\Z}(1 - t_k)}
    \frac{ d^{\nu_{W,\ell}-1} }{ dz^{\nu_{W,\ell}-1} }
        \left.\left( \frac{ 1 }
            {\prod\limits_{(k,i)\in\Lambda}
                (z^{a_{k,i}} - t_k)(1 - t_k z^{a_{k,i}})}
        \right)\right|_{z=0} .
\]
Using the Cauchy product formula for the Maclaurin series of
$1/\big(\prod_{(k,i)\in\Lambda} (z^{a_{k,i}} - t_k)(1 - t_k z^{a_{k,i}})\big)$ and the series expansions
\[
    \frac{1}{1 - t_k z^{a_{k,i}}}
    =
    \sum\limits_{j = 0}^\infty t_k^{j} z^{j a_{k,i}}
    \quad\quad\mbox{and}\quad\quad
    \frac{1}{z^{a_{k,i}} - t_k}
    =
    \sum\limits_{j = 0}^\infty - \frac{ z^{j a_{k,i}} }{t_k^{j+1}},
\]
the residue is given by summing products of the form $\prod_{(k,i)\in\Lambda} -t_k^{I_{k,i} }/t^{J_{k,i}+1 }$
where the $I_{k,i} + J_{k,i}$ are coefficients of a linear combination of the $a_{k,i}$ that is equal to
$\nu_{W,\ell} - 1$ and the sum is over all such linear combinations. In other words, we can express this residue as
\begin{align}
\label{eq:ResidueAtZero}
    \nonumber
    \frac{1}{\prod\limits_{d_k\in 2\Z}(1 - t_k)}
    \sum\limits_{S} \prod\limits_{(k,i)\in\Lambda}
        \left( (t_k^{I_{k,i}}) \left(- \frac{1}{t_k^{J_{k,i} + 1}} \right)\right)
    =
    \frac{(-1)^C}{\prod\limits_{d_k\in 2\Z}(1 - t_k)}
    \sum\limits_{S} \prod\limits_{(k,i)\in\Lambda}
        t_k^{I_{k,i} - J_{k,i} - 1}
\end{align}
where $S$ is the set of nonnegative integer solutions $(I_{k,i}, J_{k,i})_{(k,i)\in\Lambda}$
of the equation $\sum_{(k,i)\in\Lambda} (I_{k,i} + J_{k,i}) a_{k,i} = \nu_{W,\ell} - 1$, i.e.
\begin{equation}
\label{eq:DefS}
    S = \left\{ (I_{k,i}, J_{k,i})_{(k,i)\in\Lambda} \in\Z_{\geq 0}^{2C}
        \; : \; \sum\limits_{(k,i)\in\Lambda} (I_{k,i} + J_{k,i}) a_{k,i} = \nu_{W,\ell} - 1
        \right\} .
\end{equation}
Combining the above observations and recalling that the Hilbert series
$\Hilb_{(\C[W]\otimes V_L)^{\SL_2}}(t_1,\ldots,t_r)$ of the module of covariants
is given by $\Upsilon_{W,-L}(t_1,\ldots,t_k) - \Upsilon_{W, L+2}(t_1,\ldots,t_k)$,
we have the following.

\begin{proposition}
\label{prop:CovarMultivarHilbSer}
Let $W = \bigoplus_{k=1}^r V_{d_k}$ be an $\SL_2$-representation with $W^{\SL_2} = \{0\}$
and let $L\in\Z^+$. If $\nu_{W,L} \leq 0$, then the $\Gamma$-graded Hilbert series
$\Hilb_{(\C[W]\otimes V_L)^{\SL_2}}(t_1,\ldots,t_r)$ of the module of covariants $W \to V_L$ is given by
\begin{align}
\label{eq:CovarMultivarHilbSerNoZeroRes}
    &\Upsilon_{W,-L}(t) - \Upsilon_{W, L+2}(t)
    \\ \nonumber &=
    \sum\limits_{(K,I)\in\Lambda} \sum\limits_{\zeta^{a_{K,I}} = 1}
    \frac{ \zeta^{-L} t_K^{-L/a_{K,I}} - \zeta^{L+2} t_K^{(L+2)/a_{K,I}}}
    {a_{K,I}(1 - t_K^2) \prod\limits_{d_k\in 2\Z}(1 - t_k)
        \prod\limits_{\substack{(k,i)\in\Lambda\smallsetminus \\ \{(K,I)\}}}
            (1 - \zeta^{-a_{k,i}} t_k t_K^{-a_{k,i}/a_{K,I}})
            (1 - \zeta^{a_{k,i}} t_k t_K^{a_{k,i}/a_{K,I}})
            }.
\end{align}
If $\nu_{W,L} > 0$, the Hilbert series is given by
\begin{align}
\label{eq:CovarMultivarHilbSerZeroRes}
    &\Upsilon_{W,-L}(t) - \Upsilon_{W, L+2}(t)
    \\ \nonumber &=
    \sum\limits_{(K,I)\in\Lambda} \sum\limits_{\zeta^{a_{K,I}} = 1}
    \frac{ \zeta^{-L} t_K^{-L/a_{K,I}} - \zeta^{L+2} t_K^{(L+2)/a_{K,I}}}
    {a_{K,I}(1 - t_K^2) \prod\limits_{d_k\in 2\Z}(1 - t_k)
        \prod\limits_{\substack{(k,i)\in\Lambda\smallsetminus \\ \{(K,I)\}}}
            (1 - \zeta^{-a_{k,i}} t_k t_K^{-a_{k,i}/a_{K,I}})
            (1 - \zeta^{a_{k,i}} t_k t_K^{a_{k,i}/a_{K,I}})
            }
    \\ \nonumber &
    \frac{(-1)^C}{\prod\limits_{d_k\in 2\Z}(1 - t_k)}
    \sum\limits_{S} \prod\limits_{(k,i)\in\Lambda}
        t_k^{I_{k,i} - J_{k,i} - 1},
\end{align}
where $S$ is defined in Equation~\eqref{eq:DefS}.
\end{proposition}

\begin{remark}
\label{rem:CovarMiltivarHilbSerSimplified}
Using the complete list $\Theta$ of weights and approximating $\bs{a}_{\Theta}$ with real parameters
$\bs{b}_{\Theta} := \big(b_{k,i} : (k,i)\in\Theta\big)$, one can express
Equation~\eqref{eq:CovarMultivarHilbSerNoZeroRes} (and hence the first line of
Equation~\eqref{eq:CovarMultivarHilbSerZeroRes}) more succinctly as
\[
    \sum\limits_{(K,I)\in\Lambda} \sum\limits_{\zeta^{a_{K,I}} = 1}
    \frac{ \zeta^{-L} t_K^{-L/a_{K,I}} - \zeta^{L+2} t_K^{(L+2)/a_{K,I}}}
    {a_{K,I} \prod\limits_{\substack{(k,i)\in\Theta\smallsetminus \\ \{(K,I)\}}}
            (1 - \zeta^{-a_{k,i}} t_k t_K^{-a_{k,i}/a_{K,I}})
            }.
\]
\end{remark}

\begin{remark}
\label{cor:ParityVanishing}
Let $W = \bigoplus_{k=1}^r V_{d_k}$ be a unitary $\SL_2$-representation with
$W^{\SL_2} = \{0\}$ and let $L\in\Z^+$. If each $d_k$ is even and $L$ is odd, then
by the Clebsch-Gordan decomposition, the $\SL_2$-representation $\C[W]\otimes V_L$ has no invariants, and hence
$\Hilb_{(\C[W]\otimes V_L)^{\SL_2}}(t_1,\ldots,t_r) = 0$.
This can also easily be concluded from Equations~\eqref{eq:CovarMultivarHilbSerNoZeroRes}
and \eqref{eq:CovarMultivarHilbSerZeroRes} by noting that in this case, the terms in the sum over
$\zeta$ in Equation~\eqref{eq:CovarMultivarHilbSerNoZeroRes} occur in positive and negative pairs,
and by a parity argument, the set $S$ in Equation~\eqref{eq:CovarMultivarHilbSerZeroRes} is empty.
\end{remark}


\subsection{The Univariate Hilbert Series}
\label{subsec:CovarHilbSerUnivar}

The univariate Hilbert series $\Hilb_{(\C[W]\otimes V_L)^{\SL_2}}(t)$ is evidently given by the
substitution $\Hilb_{(\C[W]\otimes V_L)^{\SL_2}}(t) = \Hilb_{(\C[W]\otimes V_L)^{\SL_2}}(t,\ldots,t)$
in Equations~\eqref{eq:CovarMultivarHilbSerNoZeroRes} and \eqref{eq:CovarMultivarHilbSerZeroRes}. However,
in Equations~\eqref{eq:CovarMultivarHilbSerNoZeroRes}, this yields
\[
    \sum\limits_{(K,I)\in\Lambda} \sum\limits_{\zeta^{a_{K,I}} = 1}
    \frac{ \zeta^{-L} t^{-L/a_{K,I}} - \zeta^{L+2} t^{(L+2)/a_{K,I}}}
    {a_{K,I}(1 - t^2) (1 - t)^e
        \prod\limits_{\substack{(k,i)\in\Lambda\smallsetminus \\ \{(K,I)\}}}
            (1 - \zeta^{-a_{k,i}} t^{(a_{K,I} - a_{k,i})/a_{K,I}})
            (1 - \zeta^{a_{k,i}} t^{(a_{K,I} + a_{k,i})/a_{K,I}})
            },
\]
where $e$ denotes the number of even $d_k$. When two of the $d_k$ have the same parity, this
expression fails to be defined; for example, for terms where $\zeta = 1$ and
$a_{K,I} = a_{k,i}$, we have $1 - \zeta^{-a_{k,i}} t^{(a_{K,I} - a_{k,i})/a_{K,I}} = 0$.
However, it was explained in \cite[Section 3.2]{CayresPintoHerbigHerdenSeaton} for the case
$L = 0$ that the singularities that occur in these degenerate cases are removable; an identical
procedure applies in this case and was also applied in \cite[Section 3.2]{HerbigSeaton} and
\cite[Theorem 3.3]{CowieHerbigSeatonHerden}. Note that no such singularities arise in the
additional terms in Equation~\eqref{eq:CovarMultivarHilbSerZeroRes}.

Specifically, suppose $a$ is a positive value of
$a_{i,j}$ that occurs $N$ times, and let $\Lambda^a := \{(k,i)\in\Lambda \mid a_{k,i}\neq a \}$.
We choose $x_j$ for $j = 1,\ldots,N$ and $x_{k,i}$ for $(k,i)\in\Lambda^a$ to be distinct elements
of the interior of the unit disk. One then considers the integral
\[
    \frac{1}{2\pi\sqrt{-1}}\int\limits_{\Sp^1}
        \frac{ z^{\ell-1} \, dz}
        {(1 - t)^e \prod\limits_{j=1}^N (1 - x_j z^{-a})(1 - x_j z^{a})
            \prod\limits_{(k,i)\in\Lambda^a}
            (1 - t x_{k,i}^{-a_{k,i}})(1 - x_{k,i} z^{a_{k,i}})},
\]
which can be computed as in Subsection~\ref{subsec:CovarHilbSerMultivar}. The singularities
appear in the form of factors $x_p - x_q$ in the denominator, and the sum of residues for
a fixed $\zeta$ is divisible by the full Vandermonde determinant
$\prod_{1\leq p < q \leq N} x_p - x_q$. It is then easy to verify that the numerator is an
alternating polynomial in the $x_j$, implying that these singularities are removable.
Therefore, approximating the $a_{k,i}$ with distinct real parameters $b_{k,i}$ and letting
$\bs{b} := \big(b_{k,i} : (k,i)\in\Lambda\big)$, we can express $\Hilb_{(\C[W]\otimes V_L)^{\SL_2}}(t)$
as the limit $\bs{b}\to\bs{a}$ as follows.

\begin{theorem}
\label{thrm:CovarUnivarHilbSer}
Let $W = \bigoplus_{k=1}^r V_{d_k}$ be a unitary $\SL_2$-representation with $W^{\SL_2} = \{0\}$
and let $L\in\Z^+$. If $\nu_{W,L} \leq 0$, then the Hilbert series $\Hilb_{(\C[W]\otimes V_L)^{\SL_2}}(t)$
of the module of covariants $W \to V_L$ is given by
\begin{equation}
\label{eq:CovarUnivarHilbSerNoZeroRes}
    \lim\limits_{\bs{b}\to\bs{a}}
    \sum\limits_{(K,I)\in\Lambda} \sum\limits_{\zeta^{a_{K,I}} = 1}
    \frac{ \zeta^{-L} t^{-L/b_{K,I}} - \zeta^{L+2} t^{(L+2)/b_{K,I}}}
    {b_{K,I}(1 - t^2) (1 - t)^e
        \prod\limits_{\substack{(k,i)\in\Lambda\smallsetminus \\ \{(K,I)\}}}
            (1 - \zeta^{-a_{k,i}} t^{(b_{K,I} - b_{k,i})/b_{K,I}})
            (1 - \zeta^{a_{k,i}} t^{(b_{K,I} + b_{k,i})/b_{K,I}})
            }.
\end{equation}
If $\nu_{W,L} > 0$, the Hilbert series is given by
\begin{align}
\label{eq:CovarUnivarHilbSerZeroRes}
    &\lim\limits_{\bs{b}\to\bs{a}}
    \sum\limits_{(K,I)\in\Lambda} \sum\limits_{\zeta^{a_{K,I}} = 1}
    \frac{ \zeta^{-L} t^{-L/b_{K,I}} - \zeta^{L+2} t^{(L+2)/b_{K,I}}}
    {b_{K,I}(1 - t^2) (1 - t)^e
        \prod\limits_{\substack{(k,i)\in\Lambda\smallsetminus \\ \{(K,I)\}}}
            (1 - \zeta^{-a_{k,i}} t^{(b_{K,I} - b_{k,i})/b_{K,I}})
            (1 - \zeta^{a_{k,i}} t^{(b_{K,I} + b_{k,i})/b_{K,I}})
            }
    \\ \nonumber & + \frac{(-1)^C}{(1 - t)^e}
    \sum\limits_{S} \prod\limits_{(k,i)\in\Lambda}
        t^{I_{k,i} - J_{k,i} - 1},
\end{align}
where $S$ is defined in Equation~\eqref{eq:DefS}.
\end{theorem}

\begin{remark}
\label{rem:CovarUnivarHilbSerSimplified}
As in the case of Remark~\ref{rem:CovarMiltivarHilbSerSimplified},
we can express Equation~\eqref{eq:CovarUnivarHilbSerNoZeroRes} and the first line of
Equation~\eqref{eq:CovarUnivarHilbSerZeroRes} more succinctly as
\[
    \lim\limits_{\bs{b}\to\bs{a}}
    \sum\limits_{(K,I)\in\Lambda} \sum\limits_{\zeta^{a_{K,I}} = 1}
    \frac{ \zeta^{-L} t^{-L/b_{K,I}} - \zeta^{L+2} t^{(L+2)/b_{K,I}}}
    {b_{K,I} \prod\limits_{\substack{(k,i)\in\Theta\smallsetminus \\ \{(K,I)\}}}
            (1 - \zeta^{-a_{k,i}} t^{(b_{K,I} - b_{k,i})/b_{K,I}})
            }.
\]
\end{remark}

As we will see below, expressing the Hilbert series $\Hilb_{(\C[W]\otimes V_L)^{\SL_2}}(t)$ as a limit as
in Theorem~\ref{thrm:CovarUnivarHilbSer} will allow us to compute the first few Laurent coefficients in general.
Moreover, the expressions in Theorem~\ref{thrm:CovarUnivarHilbSer} can be used to determine an algorithm to compute
the Hilbert series. In the \emph{generic} case, i.e. when the weights $a_{k,i}$ are
distinct (which implies that $r \leq 2$ and, when $r = 2$, the $d_i$ have opposite parities),
the limit is unnecessary. The algorithm is accomplished by
using the operator $U_a$ on formal power series that assigns to $F(t) = \sum_{i=0}^\infty F_i t^i$ the series
\[
    (U_a F)(t) :=   \sum\limits_{i=0}^\infty F_{ia} t^i.
\]
When $F(t)$ is the power series of a rational function, it is easy to see that
\[
    (U_a F)(t)  =   \frac{1}{a} \sum\limits_{\zeta^a = 1} F(\zeta \sqrt[a](t)),
\]
so that $U_a$ can be used to compute the sums over roots of unity in
Equations~\eqref{eq:CovarUnivarHilbSerNoZeroRes} and \eqref{eq:CovarUnivarHilbSerZeroRes}.
In the \emph{degenerate} case, where
the $a_{k,i}$ are not distinct and so that the limit is required, a partial fraction decomposition can be
used to remove the singularities before applying a similar procedure.

This algorithm was described in detail for the case $L = 0$ in \cite[Section 6]{CayresPintoHerbigHerdenSeaton}.
The case of Equation~\eqref{eq:CovarUnivarHilbSerNoZeroRes} has only slight modifications, so we refer the reader
to that reference for more details. For the case of Equation~\eqref{eq:CovarUnivarHilbSerZeroRes}, one need only
compute the second line, the residue at $z=0$, which is easy to implement directly. We have implemented this
algorithm on \emph{Mathematica} \cite{Mathematica}, and it is available from the authors upon request.


\subsection{The Laurent Coefficients of the Hilbert series}
\label{subsec:CovarHilbSerLaurent}

In this section, we compute the first three Laurent coefficients of the Hilbert series
$\Hilb_{(\C[W]\otimes V_L)^{\SL_2}}(t)$ of covariants. First, let us establish some notation.
Recall that $D$ denotes the dimension of $W$, and $\C[W]^{\SL_2}$ has Krull dimension $3 - D$
unless $(W,\SL_2)$ fails to be $1$-large, i.e. unless $W$ is isomorphic to $V_1$, $2V_1$,
or $V_2$, see \cite[Remark 9.2(3)]{GWSlifting} and \cite[Theorem 3.4]{HerbigSchwarz}.
We use $\gamma_{m,L}(W)$ for the Laurent coefficients, so that
\[
    \Hilb_{(\C[W]\otimes V_L)^{\SL_2}}(t)
    =   \sum\limits_{m=0}^\infty \gamma_{m, L}(W) (1 - t)^{m - \dim \C[W]^{\SL_2}}.
\]
We will often abbreviate $\gamma_{m,L}(W)$ as $\gamma_{m,L}$ when $W$ is clear from the context.
Following \cite{CayresPintoHerbigHerdenSeaton}, we let $\gamma_m(W) := \gamma_{m,0}(W)$ denote the Laurent
coefficients of the invariants. Note in particular that we index the $\gamma_{m,L}(W)$ to match the degrees
of the $\gamma_m(W)$, i.e. $\gamma_{m,L}(W)$ will denote the degree $3-D+m$ coefficient even if this implies
$\gamma_{0,L}(W) = 0$. We will see below that the pole order of $\Hilb_{(\C[W]\otimes V_L)^{\SL_2}}(t)$
is bounded by that of $\C[W]^{\SL_2}$, which was also observed in \cite{BroerCovariants}.

In order to take advantage of the computations of the $\gamma_m$ in \cite{CayresPintoHerbigHerdenSeaton},
we use a slightly different formulation of the Hilbert series than that given in Theorem~\ref{thrm:CovarUnivarHilbSer}.
Define the function
\[
    H_{W,K,I,\zeta}^{\ell}(\bs{b}_{\Theta},t)
    =
    \frac{ t^{\ell/b_{K,I}} \zeta^{\ell} \big(1 - \zeta^2 t^{2/b_{K,I}} \big)}
    {b_{K,I} \prod\limits_{\substack{(k,i)\in\Theta \smallsetminus \\ \{(K,I)\}}}
            (1 - \zeta^{-a_{k,i}} t^{(b_{K,I}-b_{k,i})/b_{K,I}}) }.
\]
By telescoping the numerator in the sum over $\ell$, it is easy to see that
when $\nu_{W,L} \leq 0$, the Hilbert series is given by
\[
    \Hilb_{(\C[W]\otimes V_L)^{\SL_2}}(t)
    =
    \lim\limits_{\bs{b}\to\bs{a}}
    \sum\limits_{\ell=0}^L
    \sum\limits_{(K,I)\in\Lambda} \sum\limits_{\zeta^{a_{K,I}} = 1}
    H_{W,K,I,\zeta}^{2\ell-L}(\bs{b}_{\Theta},t).
\]
Note that $H_{W,K,I,\zeta}^0(\bs{b}_{\Theta},t)$ was denoted $H_{W,K,I,\zeta}(\bs{b}_{\Theta},t)$ in
\cite{CayresPintoHerbigHerdenSeaton}. Using this notation,
\[
    H_{W,K,I,\zeta}^{\ell}(\bs{b}_{\Theta},t)
    =
    t^{\ell/b_{K,I}} \zeta^{\ell} H_{W,K,I,\zeta}(\bs{b}_{\Theta},t),
\]
so that we can express the Hilbert series as
\begin{equation}
\label{eq:HilbCovarFullSum}
    \Hilb_{(\C[W]\otimes V_L)^{\SL_2}}(t)
    =
    \lim\limits_{\bs{b}\to\bs{a}}
    \sum\limits_{\ell=0}^L  \sum\limits_{(K,I)\in\Lambda} t^{(2\ell-L)/b_{K,I}}
        \sum\limits_{\zeta^{a_{K,I}} = 1}
            \zeta^{2\ell-L} H_{W,K,I,\zeta}(\bs{b}_{\Theta},t).
\end{equation}
In particular, note that as $t^{(2\ell-L)/b_{K,I}}$ is holomorphic at $t = 1$, the pole orders
of the covariants are bounded by the pole orders of the invariants as noted above.
Let $\gamma_m(W,K,I,\zeta)$ denote the contribution to $\gamma_m$ of the term
$H_{W,K,I,\zeta}(\bs{b}_{\Theta},t)$, i.e. the coefficient of degree $3-D+m$.

\begin{remark}
\label{rem:CovarExceptions}
The computation of the $\gamma_{m,L}$ for $m \leq 2$ can be treated uniformly except for a handful
of low-dimensional $W$, which can easily be computed individually using the algorithm described in
Section~\ref{subsec:CovarHilbSerUnivar}. For the representations $V_1$ and $V_2$, the pole order at $t = 1$
is not equal to $D - 3$; see \cite[Table 1]{CayresPintoHerbigHerdenSeaton}.
Other exceptions arise because a term $H_{W,K,I,\zeta}^{\ell}(\bs{b}_{\Theta},t)$ with $\zeta \neq \pm 1$
contributes to $\gamma_{m,L}$ for $m \leq 2$, which only occurs in small dimensions, or because the terms
$H_{W,K,I,\zeta}^{\ell}(\bs{b}_{\Theta},t)$ fail to have enough factors in the denominator for the general
arguments to apply. We refer the reader to \cite[Section 4.1]{CayresPintoHerbigHerdenSeaton} for a
careful discussion of the reason for each exception; as we use many of the computations from that
reference, the reasoning remains the same. Here, we only summarize that
the exceptions for the computation of $\gamma_{0,L}$ are $V_d$ for $d = 1,2,3,4$ and $2V_1$;
the exceptions for the computation of $\gamma_{1,L}$ are $V_d$ for $d = 1,2,3,4$, $2V_1$, $V_1 \oplus  V_2$,
and $2V_2$ (see below); and
the exceptions for the computation of $\gamma_{2,L}$ are $V_d$ for $d = 1,2,3,4,5,6,8$, $2V_1$, $V_1 \oplus  V_2$,
$V_1 \oplus  V_3$, $V_1 \oplus  V_4$, $2V_2$, $V_2 \oplus  V_3$, $V_2 \oplus  V_4$, $2V_3$, and $2V_4$.
For our primary interest in computing the first Laurent coefficient of $\Hilb_V^{\on}(t)$, only
five of these exceptions are relevant and given in Table~\ref{tab:HilbM0LowDim}; see Section~\ref{sec:SympQuots}.

We may also ignore the additional sum that appears in Equation~\eqref{eq:CovarUnivarHilbSerZeroRes}.
Note that this sum has a pole at $t = 1$ of order of $e$,
and hence contributes to $\gamma_{m,L}$ if and only if $D \leq e + m + 3$.
For $m\leq 2$, one easily checks that this sum contributes to $\gamma_{m,L}$
only in cases on the list of exceptions for $\gamma_{m,L}$ above; this is the reason we included
$2V_2$ on the exception list for $\gamma_{1,L}$, which was not an exception for $\gamma_1$ in
\cite[Theorem 1.1]{CayresPintoHerbigHerdenSeaton}.
\end{remark}

Using the series expansion
\begin{align*}
    t^{(2\ell-L)/b_{K,I}}
    &=  1 + \frac{L - 2\ell}{b_{K,I}}(1 - t) + \frac{(L - 2\ell)(b_{K,I} + L - 2\ell)}{2b_{K,I}^2}(1 - t)^2
    \\&\quad\quad
        + \frac{(L - 2\ell)(b_{K,I} + L - 2\ell)(2b_{K,I} + L - 2\ell)}
            {6 b_{K,I}^3}(1 - t)^3
    \\&\quad\quad
        + \frac{(L - 2\ell)(b_{K,I} + L - 2\ell)(2b_{K,I} + L - 2\ell)(3 b_{K,I} + L - 2\ell)}
            {24 b_{K,I}^4}(1 - t)^4
    \\&\quad\quad
        + \frac{(L - 2\ell)(b_{K,I} + L - 2\ell)(2b_{K,I} + L - 2\ell)(3 b_{K,I} + L - 2\ell)(4 b_{K,I} + L - 2\ell)}
            {120 b_{K,I}^5}(1 - t)^5
        + \cdots
\end{align*}
and Equation~\eqref{eq:HilbCovarFullSum}, we can express $\Hilb_{(\C[W]\otimes V_L)^{\SL_2}}(t)$ as the product
\begin{align}
\label{eq:HilbCovarFullSumExpanded}
    &\Hilb_{(\C[W]\otimes V_L)^{\SL_2}}(t)
    =
    \lim\limits_{\bs{b}\to\bs{a}}
    \sum\limits_{\ell=0}^L  \sum\limits_{(K,I)\in\Lambda}
    \Big( 1 + \frac{L - 2\ell}{b_{K,I}}(1 - t) + \frac{(L - 2\ell)(b_{K,I} + L - 2\ell)}{2b_{K,I}^2}(1 - t)^2
        + \cdots \Big)\cdot
    \\ \nonumber
    &\sum\limits_{\zeta^{a_{K,I}} = 1}
            \zeta^{2\ell-L}
            \Big( \gamma_0(W,K,I,\zeta)(1 - t)^{3 - D} + \gamma_1(W,K,I,\zeta)(1 - t)^{4 - D}
                + \gamma_2(W,K,I,\zeta)(1 - t)^{5 - D} + \cdots \Big).
\end{align}
We will use this expression to compute the Laurent coefficients of
$\Hilb_{(\C[W]\otimes V_L)^{\SL_2}}(t)$ and express them in terms of Schur polynomials.

For an integer partition
$\rho = (\rho_1,\ldots,\rho_n) \in \Z^n$ with $\rho_1\geq\rho_2\geq\cdots\rho_n\geq 0$, let
$s_\rho(\bs{x})$ denote the corresponding Schur polynomial in the variables $\bs{x}=(x_1,\ldots,x_n)$, i.e.
\[
    s_\rho(\bs{x})  =   \frac{\det\left( x_i^{ \rho_j + n - i} \right)}
                            {\det\left( x_i^{ n - i} \right)}.
\]
See \cite[I.3]{MacdonaldSymFuncs} or \cite[Section 4.6]{SaganSymGrp} for more details.
Note that we will sometimes for convenience
refer to $s_\rho(\bs{x})$ where the condition $\rho_1\geq\rho_2\geq\cdots\rho_n\geq 0$ does not hold; these
are defined in the same way but may yield Laurent-Schur polynomials or vanish.
We will often use the shorthand $s_m(\bs{x})$ to denote the Schur polynomial $s_{m,n-2,n-3,\ldots,1,0}(\bs{x})$.


\begin{theorem}
\label{thrm:CovarGamma0}
Let $W = \bigoplus_{k=1}^r V_{d_k}$ be a unitary $\SL_2$-representation with $W^{\SL_2} = \{0\}$
and let $L\in\Z^+$, and assume that $W$ is not isomorphic to $V_d$ for $d \leq 4$ nor $2V_1$.
If at least one $d_k$ is odd, the degree $3-D$ coefficient $\gamma_{0,L}$ of the Laurent series
of $\Hilb_{(\C[W]\otimes V_L)^{\SL_2}}(t)$ is given by
\[
    \gamma_{0,L}    =   (L + 1)\gamma_0
                    =   \frac{(L + 1) s_\rho(\bs{a}) }{s_\delta(\bs{a}) },
\]
where $\rho = (C-3, C-3, C-3, C-4, \ldots, 1,0)$ and $\delta = (C-1, C-2, \ldots, 1, 0)$.
If all $d_k$ are even, then
\[
    \gamma_{0,L}    =   \frac{\big(1 + (-1)^L\big)(L + 1)\gamma_0 }{2}
                    =   \frac{\big(1 + (-1)^L\big)(L + 1) s_\rho(\bs{a}) }{s_\delta(\bs{a}) }.
\]
\end{theorem}
Note in particular that when all $d_k$ are even and $L$ is odd, $\gamma_{0,L} = 0$; in fact
$\Hilb_{(\C[W]\otimes V_L)^{\SL_2}}(t) = 0$ in this case by Remark~\ref{cor:ParityVanishing}.
\begin{proof}
From Equation~\eqref{eq:HilbCovarFullSum}, we have
\[
    \gamma_{0,L}
    =
    \lim\limits_{\bs{b}\to\bs{a}}
        \sum\limits_{\ell=0}^L  \sum\limits_{(K,I)\in\Lambda}
        \sum\limits_{\zeta^{a_{K,I}} = 1}
            \zeta^{2\ell-L} \gamma_0(W,K,I,\zeta).
\]
As explained in \cite[Section 4.1]{CayresPintoHerbigHerdenSeaton}, excluding the noted exceptions,
if any $d_k$ are odd, then $\gamma_0(W,K,I,\zeta) = 0$ unless $\zeta = 1$, so that
\[
    \gamma_0 = \lim\limits_{\bs{b}\to\bs{a}} \sum\limits_{(K,I)\in\Lambda} \gamma_0(W,K,I,1).
\]
If all $d_k$ are even, then $\gamma_0(W,K,I,-1) = \gamma_0(W,K,I,1)$ and all $\gamma_0(W,K,I,\zeta) = 0$
for $\zeta\neq\pm 1$, i.e.
\[
    \gamma_0
    =   \lim\limits_{\bs{b}\to\bs{a}} \sum\limits_{(K,I)\in\Lambda}
            \big( \gamma_0(W,K,I,1) + \gamma_0(W,K,I,-1) \big)
    =   \lim\limits_{\bs{b}\to\bs{a}} 2 \sum\limits_{(K,I)\in\Lambda}
            \gamma_0(W,K,I,1).
\]
Hence, when at least one $d_k$ is odd,
\[
    \gamma_{0,L}
    =
    \lim\limits_{\bs{b}\to\bs{a}}
        \sum\limits_{\ell=0}^L  \sum\limits_{(K,I)\in\Lambda} \gamma_0(W,K,I,1)
    =
    \sum\limits_{\ell=0}^L \gamma_0
    =
    (L + 1) \gamma_0.
\]
When all $d_k$ are even,
\begin{align*}
    \gamma_{0,L}
    &=
    \lim\limits_{\bs{b}\to\bs{a}}
        \sum\limits_{\ell=0}^L  \sum\limits_{(K,I)\in\Lambda}
        \big( \gamma_0(W,K,I,1) + (-1)^{2\ell-L}\gamma_0(W,K,I,-1) \big)
    \\&=
    \frac{1}{2} \sum\limits_{\ell=0}^L  \gamma_0 + (-1)^L \gamma_0
    \quad\quad\quad =
    \frac{\big( 1 + (-1)^L \big)\big( L + 1 \big)\gamma_0}{2}.
\end{align*}
The expressions for $\gamma_0$ in terms of Schur polynomials are given in
\cite[Theorem 1.1]{CayresPintoHerbigHerdenSeaton}, completing the proof.
\end{proof}


\begin{theorem}
\label{thrm:CovarGamma1}
Let $W = \bigoplus_{k=1}^r V_{d_k}$ be a unitary $\SL_2$-representation with $W^{\SL_2} = \{0\}$
and $d_1\leq d_2\leq\cdots\leq d_r$, and assume that $W$ is not isomorphic to $V_d$ for $d \leq 4$,
$2V_1$, $V_1\oplus V_2$, nor $2V_2$.
Let $L\in\Z^+$. If at least one $d_k$ is odd, the degree $4-D$ coefficient
$\gamma_{1,L}$ of the Laurent series of $\Hilb_{(\C[W]\otimes V_L)^{\SL_2}}(t)$ is given by
\[
    \gamma_{1,L}    =   (L + 1)\gamma_1
                    =   \frac{3(L + 1) s_\rho(\bs{a}) }{2 s_\delta(\bs{a}) },
\]
where $\rho = (C-3, C-3, C-3, C-4, \ldots, 1,0)$ and $\delta = (C-1, C-2, \ldots, 1, 0)$.
If all $d_k$ are even, then
\[
    \gamma_{1,L}    =   \frac{\big(1 + (-1)^L\big)(L + 1)\gamma_1 }{2}
                    =   \frac{3\big(1 + (-1)^L\big)(L + 1) s_\rho(\bs{a}) }{2 s_\delta(\bs{a}) }.
\]
\end{theorem}
\begin{proof}
From Equation~\eqref{eq:HilbCovarFullSumExpanded}, we express
\[
    \gamma_{1,L}
    =
    \lim\limits_{\bs{b}\to\bs{a}}
        \sum\limits_{\ell=0}^L  \sum\limits_{(K,I)\in\Lambda}
        \sum\limits_{\zeta^{a_{K,I}} = 1}
            \zeta^{2\ell-L}
            \Big( \frac{L - 2\ell}{b_{K,I}} \gamma_0(W,K,I,\zeta) + \gamma_1(W,K,I,\zeta)\Big).
\]
We consider three cases.


\noindent\textbf{Case I:} Assume that at least two $d_k$ are odd or one $d_k > 1$ is odd. Then for $m=0,1$,
we have $\gamma_m(W,K,I,\zeta) = 0$ unless $\zeta = 1$ so that
\[
    \gamma_m = \lim\limits_{\bs{b}\to\bs{a}} \sum\limits_{(K,I)\in\Lambda} \gamma_m(W,K,I,1),
    \quad\quad
    m = 0, 1.
\]
Hence,
\begin{align*}
    \gamma_{1,L}
    &=
    \lim\limits_{\bs{b}\to\bs{a}} \sum\limits_{\ell=0}^L  \sum\limits_{(K,I)\in\Lambda}
        \Big( \frac{L - 2\ell}{b_{K,I}} \gamma_0(W,K,I,1) + \gamma_1(W,K,I,1)\Big)
    \\&=
    (L + 1) \gamma_1
    + \lim\limits_{\bs{b}\to\bs{a}}
        \sum\limits_{(K,I)\in\Lambda} \frac{\gamma_0(W,K,I,1)}{b_{K,I}}
        \sum\limits_{\ell=0}^L (L - 2\ell)
    \\&=
    (L + 1) \gamma_1,
\end{align*}
as $\sum_{\ell=0}^L (L - 2\ell) = 0$.


\noindent\textbf{Case II:} Assume that all $d_k$ are even. Then for $m = 0, 1$, we have
$\gamma_0(W,K,I,1) = \gamma_0(W,K,I,-1)$ and all other $\gamma_0(W,K,I,\zeta) = 0$. Therefore,
\[
    \gamma_m = \lim\limits_{\bs{b}\to\bs{a}} 2\sum\limits_{(K,I)\in\Lambda} \gamma_m(W,K,I,1),
    \quad\quad
    m = 0, 1,
\]
and
\begin{align*}
    \gamma_{1,L}
    &=
    \lim\limits_{\bs{b}\to\bs{a}}
        \sum\limits_{\ell=0}^L  \sum\limits_{(K,I)\in\Lambda}
        \big( 1 + (-1)^{2\ell-L} \big)
            \Big( \frac{L - 2\ell}{b_{K,I}} \gamma_0(W,K,I,1) + \gamma_1(W,K,I,1)\Big)
    \\&=
    \lim\limits_{\bs{b}\to\bs{a}}
    \big( 1 + (-1)^{2\ell-L} \big)(L+1)\sum\limits_{(K,I)\in\Lambda} \gamma_1(W,K,I,1)
        +
        \sum\limits_{(K,I)\in\Lambda} \frac{\gamma_0(W,K,I,1)}{b_{K,I}}
        \sum\limits_{\ell=0}^L\big( L - 2\ell \big)
    \\&=
    \frac{\big( 1 + (-1)^L \big)\big( L + 1 \big)\gamma_1 }{2} ,
\end{align*}
again as $\sum_{\ell=0}^L \big( L - 2\ell \big) = 0$.


\noindent\textbf{Case III:} Assume $d_1 = 1$ and $d_k$ is even for $k > 1$. Then $\gamma_0(W,K,I,\zeta) = 0$
for $\zeta \neq 1$, $\gamma_1(W,K,I,\zeta) = 0$ for $\zeta \neq \pm 1$,
but $\gamma_1(W,K,I,1)$ and $\gamma_1(W,K,I,-1)$ are not necessarily equal. We compute
\begin{align*}
    \gamma_{1,L}
    &=
    \lim\limits_{\bs{b}\to\bs{a}}
    \sum\limits_{\ell=0}^L  \sum\limits_{(K,I)\in\Lambda}
            \Big( \frac{L - 2\ell}{b_{K,I}} \gamma_0(W,K,I,1) + \gamma_1(W,K,I,1)
            + (-1)^{2\ell-L}\gamma_1(W,K,I,-1)
            \Big)
    \\&=
    \lim\limits_{\bs{b}\to\bs{a}}
    \sum\limits_{\ell=0}^L \Big( \sum\limits_{(K,I)\in\Lambda} \gamma_1(W,K,I,1) + (-1)^{L}\gamma_1(W,K,I,-1) \Big)
    + \sum\limits_{(K,I)\in\Lambda}
        \frac{\gamma_0(W,K,I,1)}{b_{K,I}} \sum\limits_{\ell=0}^L L - 2\ell
    \\&=
    \lim\limits_{\bs{b}\to\bs{a}}
    (L + 1)\Big( \sum\limits_{(K,I)\in\Lambda} \gamma_1(W,K,I,1)
        + (-1)^{L} \sum\limits_{(K,I)\in\Lambda} \gamma_1(W,K,I,-1) \Big).
\end{align*}
Now, in the proof of \cite[Theorem 5.5]{CayresPintoHerbigHerdenSeaton},
it is demonstrated that $\sum_{(K,I)\in\Lambda} \gamma_1(W,K,I,-1) = 0$, hence
\[
    \gamma_{1,L}
    =
    \lim\limits_{\bs{b}\to\bs{a}}
    (L + 1)\Big( \sum\limits_{(K,I)\in\Lambda} \gamma_1(W,K,I,1) \Big)
    =
    (L + 1)\gamma_1.
    \qedhere
\]
\end{proof}


\begin{theorem}
\label{thrm:CovarGamma2}
Let $W = \bigoplus_{k=1}^r V_{d_k}$ be a unitary $\SL_2$-representation with $W^{\SL_2} = \{0\}$
and $d_1\leq d_2\leq\cdots\leq d_r$, and assume that $W$ is not isomorphic to
$V_d$ for $d = 1,2,3,4,5,6,8$, $2V_1$, $V_1\oplus V_2$, $V_1\oplus V_3$, $V_1\oplus V_4$, $2V_2$, $V_2\oplus V_3$, $V_2\oplus V_4$,
$2V_3$, nor $2V_4$. Let $L\in\Z^+$.

If at least two $d_k$ are odd or one $d_k > 1$ is odd, the degree $5-D$ coefficient
$\gamma_{2,L}$ of the Laurent series of $\Hilb_{(\C[W]\otimes V_L)^{\SL_2}}(t)$ is given by
\[
    \gamma_{2,L}
    =
    (L + 1) \gamma_2
        - \frac{ L(L + 1)(L + 2)s_{\rho^\prime}(\bs{a}) }{ 6 s_\delta(\bs{a}) },
\]
where $\rho^\prime = (C-3, C-4, C-4, C-4, C-5, \ldots, 1,0)$ and $\delta = (C-1, C-2, \ldots, 1, 0)$.
If all $d_k$ are even, then
\[
    \gamma_{2,L}
    =
    \big( 1 + (-1)^L \big)
    \left( \frac{L + 1}{2} \gamma_2
        - \frac{ L(L + 1)(L + 2)s_{\rho^\prime}(\bs{a}) }{ 6 s_\delta(\bs{a}) } \right).
\]
If $d_1 = 1$ and each $d_k$ is even for $k > 1$, then
\begin{align*}
    \gamma_{2,L}
    &=
    (L + 1)\left(\frac{42 s_{\rho}(\bs{a}) + s_{\rho^\prime}(\bs{a})
        \big( P_2(\bs{a})  - 8 - 4 L(L + 2)\big) }
            {24 s_\delta(\bs{a})}
        + (-1)^{L} \frac{s_{C - 4, C - 4, C - 4, C - 5, \ldots, 1, 0}(\bs{a_1}) }
            {4 s_{C - 2, C - 3, C - 4 \ldots, 1, 0}(\bs{a_1}) } \right),
\end{align*}
where $\rho = (C-3, C-3, C-3, C-4, \ldots, 1,0)$, $P_2$ is the quadratic power sum, and
$\bs{a}_1$ denotes $\bs{a}$ with entry $a_{1,1}$ removed.
\end{theorem}
\begin{proof}
Again using Equation~\eqref{eq:HilbCovarFullSumExpanded}, we can express $\gamma_{2,L}$ as
\[
    \lim\limits_{\bs{b}\to\bs{a}}
        \sum\limits_{\ell=0}^L  \sum\limits_{(K,I)\in\Lambda}
        \sum\limits_{\zeta^{a_{K,I}} = 1}
            \zeta^{2\ell-L}
            \left( \frac{(L - 2\ell)(b_{K,I} + L - 2\ell)}{2b_{K,I}^2} \gamma_0(W,K,I,\zeta)
            + \frac{L - 2\ell}{b_{K,I}} \gamma_1(W,K,I,\zeta) + \gamma_2(W,K,I,\zeta)\right).
\]
As explained in \cite[Section 4.1]{CayresPintoHerbigHerdenSeaton}, except for the listed exceptions,
we have that $\gamma_0(W,K,I,\zeta) = \gamma_1(W,K,I,\zeta) = \gamma_2(W,K,I,\zeta) = 0$
unless $\zeta = \pm 1$. We consider three cases.


\noindent\textbf{Case I:} Assume that at least two $d_k$ are odd or one $d_k > 1$ is odd.
Then $\gamma_i(W,K,I,\zeta) = 0$ for $i = 0,1,2$ unless $\zeta = 1$, so that
\[
    \gamma_m = \lim\limits_{\bs{b}\to\bs{a}} \sum\limits_{(K,I)\in\Lambda} \gamma_m(W,K,I,1),
    \quad\quad
    m = 0, 1, 2.
\]
Therefore,
\begin{align*}
    \gamma_{2,L}
    &=
    \lim\limits_{\bs{b}\to\bs{a}}
        \sum\limits_{\ell=0}^L  \sum\limits_{(K,I)\in\Lambda}
        \left( \frac{(L - 2\ell)(b_{K,I} + L - 2\ell)}{2b_{K,I}^2} \gamma_0(W,K,I,1)
            + \frac{L - 2\ell}{b_{K,I}} \gamma_1(W,K,I,1) + \gamma_2(W,K,I,1)\right)
    \\&=
    (L + 1) \gamma_2
        + \frac{L(L + 1)(L + 2)}{6} \lim\limits_{\bs{b}\to\bs{a}}
            \sum\limits_{(K,I)\in\Lambda} \frac{\gamma_0(W,K,I,1)}{b_{K,I}^2}.
\end{align*}
We are able to compute the sum over $(K,I)\in\Lambda$ using the computations of
\cite{CayresPintoHerbigHerdenSeaton}. First, using Proposition~4.2 of that reference,
\[
    \sum\limits_{(K,I)\in\Lambda} \gamma_0(W,K,I,1)
    =
    \sum\limits_{(K,I) \in \Lambda}
        \frac{b_{K,I}^{D-3} - 2 b_{K,I}^{D-4}
            - \sum\limits_{\substack{ (\kappa,\lambda)\in\Theta\smallsetminus \\ \{(K,I)\} }}
                b_{K,I}^{D-4} b_{\kappa,\lambda} }
        {\prod\limits_{\substack{(k,i)\in\Theta\smallsetminus \\ \{(K,I)\} }}
            (b_{K,I} - b_{k,i}) } .
\]
Dividing each term by $b_{K,I}^2$ and using the functions
\[
    \Sigma_{R,S}(\bs{b}_\Theta)
        :=
        \sum\limits_{(K,I) \in \Lambda} \sum\limits_{\substack{(\kappa,\lambda)\in\Theta\smallsetminus \\ \{(K,I)\}}}
            \frac{b_{K,I}^{R} b_{\kappa,\lambda}^S }
                {\prod\limits_{\substack{(k,i)\in\Theta\smallsetminus \\ \{ (K,I)\} }}
                    (b_{K,I} - b_{k,i}) }
    \quad\quad\mbox{and}\quad\quad
    \Sigma_{R}(\bs{b}_\Theta)
        :=
        \sum\limits_{(K,I) \in \Lambda}
            \frac{ b_{K,I}^{R} }
                {\prod\limits_{\substack{(k,i)\in\Theta\smallsetminus \\ \{ (K,I)\} }}
                    (b_{K,I} - b_{k,i}) }
\]
defined in \cite[Section 5]{CayresPintoHerbigHerdenSeaton}, we have
\[
    \sum\limits_{(K,I)\in\Lambda} \frac{\gamma_0(W,K,I,1)}{b_{K,I}^2}
    =
    \Sigma_{D-5}(\bs{b}) - 2 \Sigma_{D-6}(\bs{b}) - \Sigma_{D-6,1}(\bs{b}).
\]
Then by \cite[Lemma 5.2]{CayresPintoHerbigHerdenSeaton}, specifically Equations~(34) and (35), we have
\[
    \sum\limits_{(K,I)\in\Lambda} \frac{\gamma_0(W,K,I,1)}{b_{K,I}^2}
    =
    \frac{2s_{D-e-C-5}(\bs{b})
        -
        2 s_{D-e-C-6}(\bs{b})
        -
        P_1(\bs{b}_\Theta) s_{D-e-C-6}(\bs{b})
        + s_{D-e-C-5}(\bs{b})}
        {s_\delta(\bs{b})},
\]
where $P_1$ denotes the power sum of degree $1$, and we recall the shorthand
$s_{m}(\bs{b})$ denotes $s_{m,C-2,C-3,\ldots,1,0}(\bs{b})$. As the entries of
$\bs{b}_\Theta$ are either $0$ or occur
in positive and negative pairs, $P_1(\bs{b}_\Theta) = 0$; using this and the fact that $D - e - C = C$,
\[
    \sum\limits_{(K,I)\in\Lambda} \frac{\gamma_0(W,K,I,1)}{b_{K,I}^2}
    =
    \frac{s_{C-5}(\bs{b}) - s_{C-6}(\bs{b})}{s_\delta(\bs{b})} .
\]
The non-standard Schur polynomial $s_{C - 5}(\bs{a})$ is defined in terms of the alternant associated to
$\delta+(C-5,C-2,C-3,\ldots,1,0) = (2C-6,2C-4,2C-6,\ldots,2,0)$ and hence vanishes. Rewriting the non-standard
Schur polynomial $s_{C-6}(\bs{b}) = s_{C-6,C-2,C-3,\ldots,1,0}(\bs{b})$ in standard form by permuting columns
yields $s_{C-6}(\bs{b}) = s_{C-3, C-4, C-4, C-4, C-5,\ldots,1,0}(\bs{b})$. Applying these observations completes
the proof in this case.


\noindent\textbf{Case II:} Assume that each $d_k$ is even. Then for $i = 0,1,2$,
$\gamma_i(W,K,I,\zeta) = 0$ unless $\zeta = \pm 1$, and $\gamma_i(W,K,I,1) = \gamma_i(W,K,I,-1)$.
\[
    \gamma_m = \lim\limits_{\bs{b}\to\bs{a}} 2\sum\limits_{(K,I)\in\Lambda} \gamma_m(W,K,I,1),
    \quad\quad
    m = 0, 1, 2.
\]
Hence $\gamma_{2,L}$ is given by
\begin{align*}
    &\lim\limits_{\bs{b}\to\bs{a}}
    \sum\limits_{\ell=0}^L  \sum\limits_{(K,I)\in\Lambda}
        \big( 1 + (-1)^{2\ell-L} \big)
            \left( \frac{(L - 2\ell)(b_{K,I} + L - 2\ell)}{2b_{K,I}^2} \gamma_0(W,K,I,1)
            + \frac{L - 2\ell}{b_{K,I}} \gamma_1(W,K,I,1) + \gamma_2(W,K,I,1)\right)
    \\&=
    \big( 1 + (-1)^L \big)
    \left( \frac{L + 1}{2} \gamma_2
        + \frac{L(L + 1)(L + 2)}{6} \lim\limits_{\bs{b}\to\bs{a}}
            \sum\limits_{(K,I)\in\Lambda} \frac{\gamma_0(W,K,I,1)}{b_{K,I}^2} \right).
\end{align*}
This is identical to the previous case except for the $1 + (-1)^L$ prefactor and the
change to the coefficient of $\gamma_2$, but the remaining argument is identical.


\noindent\textbf{Case III:}
Assume that $d_1 = 1$ and $d_k$ is even for $k > 1$. Then $\gamma_0(W,K,I,\zeta) = 0$ for $\zeta\neq 1$,
and for $m = 1,2$, $\gamma_m(W,K,I,\zeta) = 0$ for $\zeta \neq \pm 1$, yet
$\gamma_m(W,K,I,1)$ and $\gamma_m(W,K,I,-1)$ are not necessarily equal. Hence we can express
$\gamma_{2,L}$ as
\begin{align*}
    &
    \lim\limits_{\bs{b}\to\bs{a}}
        \sum\limits_{\ell=0}^L  \sum\limits_{(K,I)\in\Lambda}
        \sum\limits_{\zeta^{a_{K,I}} = 1}
            \zeta^{2\ell-L}
            \left( \frac{(L - 2\ell)(b_{K,I} + L - 2\ell)}{2b_{K,I}^2} \gamma_0(W,K,I,\zeta)
            + \frac{L - 2\ell}{b_{K,I}} \gamma_1(W,K,I,\zeta) + \gamma_2(W,K,I,\zeta)\right)
    \\&=
    (L+1) \lim\limits_{\bs{b}\to\bs{a}}
        \sum\limits_{(K,I)\in\Lambda} \Big( \gamma_2(W,K,I,1) + (-1)^{L} \gamma_2(W,K,I,-1)\Big)
    \\&\quad\quad\quad\quad\quad\quad
        + \lim\limits_{\bs{b}\to\bs{a}} \sum\limits_{(K,I)\in\Lambda}
            \frac{1}{b_{K,I}} \Big( \gamma_1(W,K,I,1) + (-1)^{L}  \gamma_1(W,K,I,-1) \Big)
            \sum\limits_{\ell=0}^L (L - 2\ell)
    \\&\quad\quad\quad\quad\quad\quad
        + \lim\limits_{\bs{b}\to\bs{a}} \sum\limits_{(K,I)\in\Lambda} \frac{\gamma_0(W,K,I,1)}{2b_{K,I}^2}
            \sum\limits_{\ell=0}^L (L - 2\ell)(b_{K,I} + L - 2\ell)
    \\&=
    (L + 1) \lim\limits_{\bs{b}\to\bs{a}}\sum\limits_{(K,I)\in\Lambda}
        \Big( \gamma_2(W,K,I,1) + (-1)^{L} \gamma_2(W,K,I,-1)\Big)
        +  \frac{L(L+1)(L+2)}{6} \lim\limits_{\bs{b}\to\bs{a}} \sum\limits_{(K,I)\in\Lambda} \frac{\gamma_0(W,K,I,1)}{b_{K,I}^2} .
\end{align*}
From the proof of \cite[Theorem 5.8]{CayresPintoHerbigHerdenSeaton}, we have
\[
    \sum\limits_{(K,I)\in\Lambda} \gamma_2(W,K,I,1)
    =
    \frac{42 s_{C-3,C-3,C-3,C-4,\ldots,1,0}(\bs{a}) + s_{C-3, C-4, C-4, C-4, C-5,\ldots,1,0}(\bs{a})
            \big( P_2(\bs{a})  - 8\big) }
                {24 s_\delta(\bs{a})},
\]
and
\[
    \sum\limits_{(K,I)\in\Lambda} \gamma_2(W,K,I,-1)
    =
    \frac{s_{C - 4, C - 4, C - 4, C - 5, \ldots, 1, 0}(\bs{a_1}) }
        {4 s_{C - 2, C - 3, C - 4 \ldots, 1, 0}(\bs{a_1}) },
\]
and from the computations for Case I, we have
\[
    \sum\limits_{(K,I)\in\Lambda} \frac{\gamma_0(W,K,I,1)}{b_{K,I}^2}
    =
    - \frac{s_{C-3, C-4, C-4, C-4, C-5,\ldots,1,0}(\bs{a}) }{ 2s_\delta(\bs{a}) }.
\]
Combining these observations completes the proof.
\end{proof}


\section{Hilbert Series of the Graded Algebra of Regular Functions on $\SU_2$-Symplectic Quotients}
\label{sec:SympQuots}

In this section, we use the computation of the Laurent coefficients of the
Hilbert series $\Hilb_{(\C[V\oplus V^\ast]\otimes V_2)^{\SL_2}}(t)$ in the previous section,
as well as the Laurent coefficients of $\Hilb_{\C[V\oplus V^\ast]^{\SL_2}}(t)$ computed
in \cite{CayresPintoHerbigHerdenSeaton}, to determine an explicit expression for the
first nonzero Laurent coefficient $\gamma_0^{\on}(V)$ of the Hilbert series $\Hilb_V^{\on}(t)$
of the graded algebra $\R[M_0]$ of regular functions on the symplectic quotient $M_0$
associated to $V$.

First, we note the following. If $(V,\SL_2)$ is $1$-large, i.e. not isomorphic to $V_1$, $2V_1$,
nor $V_1\oplus V_2$, then $\nu_{V,2} \leq 0$, see Equation~\eqref{eq:DefNu}. Hence,
Proposition~\ref{prop:OnShellHilb} and Theorem~\ref{thrm:CovarUnivarHilbSer} yield the following.

\begin{corollary}
\label{cor:OnShellHilb}
Let $V$ be a unitary $\SU_2$-representation with $V^{\SU_2} = \{0\}$ and assume that $V$
is not isomorphic to $V_1$, $2V_1$ nor $V_1\oplus V_2$. Then the on-shell Hilbert series
of the graded algebra of regular functions on the symplectic quotient $M_0$ is given by
\begin{equation}
\label{eq:OnShellHilbSum}
    \Hilb_V^{\on}(t)
    =
    \lim\limits_{\bs{b}\to\bs{a}}
    \sum\limits_{(K,I)\in\Lambda_{V\oplus V^\ast}} \sum\limits_{\zeta^{a_{K,I}} = 1}
    \frac{ (1 - t^6) \big(1 - \zeta^{2} t^{2/b_{K,I}}\big)
        +  (t^4 - t^2)\big(\zeta^{-2} t^{-2/b_{K,I}} - \zeta^{4} t^{4/b_{K,I}} \big)}
    {b_{K,I} \prod\limits_{\substack{(k,i)\in\Theta_{V\oplus V^\ast}\smallsetminus \\ \{(K,I)\}}}
            (1 - \zeta^{-a_{k,i}} t^{(b_{K,I} - b_{k,i})/b_{K,I}})
            },
\end{equation}
where we recall that $\Theta_{V\oplus V^\ast}$ and $\Lambda_{V\oplus V^\ast}$ denote
the sets of weights (respectively positive weights) for the cotangent lifted
representation $V\oplus V^\ast$.
\end{corollary}

Using Corollary~\ref{cor:OnShellHilb}, the algorithm described in Section~\ref{subsec:CovarHilbSerUnivar} to
compute the Hilbert series $\Hilb_{(\C[V]\otimes V_L)^{\SL_2}}(t)$ of covariants can as well be applied to compute
the Hilbert series $\Hilb_V^{\on}(t)$ of $\R[M_0]$. This amounts to applying the covariant algorithm twice to compute
both $\Hilb_{\C[V\oplus V^\ast]^{\SL_2}}(t)$ and $\Hilb_{(\C[V\oplus V^\ast]\otimes V_2)^{\SL_2}}(t)$.

Now, recall from Section~\ref{subsec:CovarHilbSerLaurent} that $\gamma_m(V\oplus V^\ast)$ denotes the Laurent coefficients
of $\C[V\oplus V^\ast]^{\SL_2}$ and $\gamma_{m,2}(V\oplus V^\ast)$ denotes the Laurent coefficients of the covariants
$V\oplus  V^\ast\to V_2$. Assume $(V,\SL_2)$ is $1$-large so that $\dim (\C[V]^{\SL_2}) = D - 3$, see
\cite[Remark (9.2)(3)]{GWSlifting}. Using the Kempf-Ness homeomorphism \cite{GWSkempfNess}, the corresponding
symplectic $M_0$ quotient has real dimension $2D - 6$. Note that the algebra of invariants $\C[V\oplus  V^\ast]^{\SL_2}$
has dimension $2D - 3$ so that $\gamma_0(V\oplus V^\ast)$ and $\gamma_{0,2}(V\oplus V^\ast)$ occur in degree $2D - 3$.

Using Proposition~\ref{prop:OnShellHilb}, the Laurent expansion of $\Hilb_V^{\on}(t)$ begins as given below,
where each $\gamma_{m,L} = \gamma_{m,L}(V\oplus V^\ast)$ and $\gamma_m = \gamma_m(V\oplus V^\ast)$:
\begin{align}
    \label{eq:OnShellExpanN1N2}
    &\big( 6\gamma_0 - 2\gamma_{0,2} \big) (1 - t)^{4 - 2D}
    + \big( - 15 \gamma_0 + 5 \gamma_{0,2} + 6 \gamma_1 - 2 \gamma_{1,2} \big) (1 - t)^{5 - 2D}
    \\ \label{eq:OnShellExpanG0} &\quad\quad
    + \big( 20 \gamma_0 - 4 \gamma_{0,2} - 15 \gamma_1 + 5 \gamma_{1,2} + 6 \gamma_2 - 2 \gamma_{2,2} \big) (1 - t)^{6 - 2D}
    \\ \label{eq:OnShellExpanG1} &\quad\quad
    + \big( -15 \gamma_0 + \gamma_{0,2} + 20 \gamma_1 - 4 \gamma_{1,2} - 15 \gamma_2 + 5 \gamma_{2,2}
        + 6 \gamma_3 - 2 \gamma_{3,2} \big) (1 - t)^{7 - 2D}
    \\ \nonumber &\quad\quad
    + \big( 6 \gamma_0 - 15 \gamma_1 + \gamma_{1,2} + 20 \gamma_2 - 4 \gamma_{2,2} - 15 \gamma_3
        + 5 \gamma_{3,2} + 6 \gamma_4 - 2 \gamma_{4,2} \big) (1 - t)^{8 - 2D}
    + \cdots .
\end{align}
As the pole order of $\Hilb_V^{\on}(t)$ is equal to the dimension of $M_0$, the first
two of these coefficients \eqref{eq:OnShellExpanN1N2} are zero. One also checks using the expressions in Theorems~\ref{thrm:CovarGamma0}
and \ref{thrm:CovarGamma1} that this is the case, i.e. that
\[
    6\gamma_0(V\oplus V^\ast) - 2\gamma_{0,2}(V\oplus V^\ast)
    =
    - 15 \gamma_0(V\oplus V^\ast) +  5 \gamma_{0,2}(V\oplus V^\ast) + 6 \gamma_1(V\oplus V^\ast) - 2 \gamma_{1,2}(V\oplus V^\ast)
    =   0.
\]

Let $\sigma_V = 2$ if each $d_k$ is even and $1$ otherwise, and note that
$\sigma_V = \sigma_{V\oplus V^\ast}$. Recall that $\bs{a}_{V\oplus V^\ast}$ denotes the vector of positive weights of
the representation $V\oplus V^\ast$ so that $\bs{a}_{V\oplus V^\ast}$ is two copies of $\bs{a}_V$ concatenated.
Let $\widehat{\delta} = (2C-1, 2C-2, \ldots, 1, 0)$,
$\widehat{\rho} = (2C-3, 2C-3, 2C-3, 2C-4, \ldots, 1,0)$, and
$\widehat{\rho^\prime} = (2C-3, 2C-4, 2C-4, 2C-4, 2C-5, \ldots, 1,0)$, i.e. the respective partitions
$\delta$, $\rho$, and $\rho^\prime$ used in Section~\ref{sec:Covariants} corresponding to the
representation $V\oplus V^\ast$. We have the following.

\begin{theorem}
\label{thrm:OnShellHilbGam0}
Let $V = \bigoplus_{k=1}^r V_{d_k}$ be a unitary $\SU_2$-representation with $V^{\SU_2} = \{0\}$ and assume
that $V$ is not isomorphic to $V_d$ for $d=1,2,3,4$ nor $V_1\oplus V_2$.
Then the first nonzero Laurent coefficient $\gamma_0^{\on}(V)$ of the Hilbert series $\Hilb_V^{\on}(t)$
of the graded algebra $\R[M_0]$ of regular functions on the symplectic quotient $M_0$ is given by
\begin{equation}
\label{eq:OnShellHilbGam0}
    \gamma_0^{\on}(V)
    =
    8 \gamma_0(V\oplus V^\ast) + \frac{ 8 \sigma_V s_{\rho^\prime}(\bs{a}_{V\oplus V^\ast}) }
        { s_\delta(\bs{a}_{V\oplus V^\ast}) }
    =
    \frac{8\sigma_V \big( s_{\rho}(\bs{a}_{V\oplus V^\ast}) + s_{\rho^\prime}(\bs{a}_{V\oplus V^\ast}) \big) }
        { s_{\delta}(\bs{a}_{V\oplus V^\ast}) }.
\end{equation}
\end{theorem}
\begin{proof}
Note that for all cases under consideration, $(V,\SL_2)$ is $1$-large, and $V\oplus V^\ast$ is not on the lists of
low-dimensional cases excluded in Theorems~\ref{thrm:CovarGamma0}, \ref{thrm:CovarGamma1}, and
\ref{thrm:CovarGamma2}. We then have by these three theorems that
\begin{align*}
    \gamma_{0,2}(V\oplus V^\ast)    &=  3\gamma_0(V\oplus V^\ast),
    \quad\quad\quad
    \gamma_{1,2}(V\oplus V^\ast)    =   3\gamma_1(V\oplus V^\ast),
    \quad\quad\mbox{and}
    \\
    \gamma_{2,2}(V\oplus V^\ast)    &=  3\gamma_2(V\oplus V^\ast)
                        - \frac{ 4\sigma_V s_{\widehat{\rho^\prime}}(\bs{a}_{V\oplus V^\ast}) }
                            { s_{\widehat{\delta}}(\bs{a}_{V\oplus V^\ast}) }.
\end{align*}
The expression for $\gamma_0^{\on}(V)$ in \eqref{eq:OnShellExpanG0} then reduces to
\begin{align*}
    \gamma_0^{\on}(V)
    =
    20 \gamma_0(V\oplus V^\ast) - 12 \gamma_0(V\oplus V^\ast) - 15 \gamma_1(V\oplus V^\ast) + 15 \gamma_1(V\oplus V^\ast)
        + 6 \gamma_2(V\oplus V^\ast) - 2 \gamma_{2,2}(V\oplus V^\ast).
\end{align*}
Equation~\eqref{eq:OnShellHilbGam0} then follows using the expressions for the $\gamma_m(V\oplus V^\ast)$
given in \cite[Theorem 1.1]{CayresPintoHerbigHerdenSeaton}.
\end{proof}

Finally, we observe that the value of $\gamma_{3,2}(V\oplus V^\ast)$ is determined by Equation~\eqref{eq:OnShellExpanG1}
and quantities computed above. Because the algebra of on-shell regular functions $\R[M_0]$ is
graded Gorenstein by \cite[Theorem 1.3]{HerbigSchwarzSeaton2}, we have by
\cite[Corollary 1.8]{HerbigHerdenSeaton} that $\gamma_1^{\on} = 0$; see also
\cite[Theorem 1.1]{HerbigHerdenSeaton2}. Then using the expression for $\gamma_1^{\on}$
in \eqref{eq:OnShellExpanG1} along with the facts that $\gamma_1(V\oplus V^\ast) = 3\gamma_0(V\oplus V^\ast)/2$ and
$\gamma_3(V\oplus V^\ast) = 5\big(\gamma_2(V\oplus V^\ast) - \gamma_0(V\oplus V^\ast)\big)/2$ from
\cite[Theorem 1.1]{CayresPintoHerbigHerdenSeaton}, we have
\[
    -15\gamma_0(V\oplus V^\ast) +  15\gamma_2(V\oplus V^\ast)
        - \frac{ 20 \sigma_V s_{\widehat{\rho^\prime}}(\bs{a}_{V\oplus V^\ast}) }
            { s_{\widehat{\delta}}(\bs{a}_{V\oplus V^\ast}) } - 2 \gamma_{3,2}(V\oplus V^\ast) = 0,
\]
i.e.
\[
    \gamma_{3,2}(V\oplus V^\ast)
    =
    \frac{-15}{2} \gamma_0(V\oplus V^\ast) + \frac{15}{2}\gamma_2(V\oplus V^\ast)
        - \frac{ 10 \sigma_V s_{\widehat{\rho^\prime}}(\bs{a}_{V\oplus V^\ast}) }
            { s_{\widehat{\delta}}(\bs{a}_{V\oplus V^\ast}) }.
\]

We conclude with some empirical observations. The data indicate that for fixed dimension of $M_0$, the irreducible representation has the smallest or second smallest $\gamma_0^{on}$ (second to $V_k \oplus V_1$ in some cases when
$\dim_\R M_0 = 0(\operatorname{mod}4$)), while the
largest $\gamma_0^{on}$ comes from the representation with the highest reducibility.
If $\dim_\R M_0=2(\operatorname{mod}4)$ the latter is $kV_1$ where $k={3\over 2}+{1\over 4}\dim_\R M_0$, while
for $\dim_\R M_0=0(\operatorname{mod} 4)$ it is $k V_1\oplus V_2$ with $k={1\over 4}\dim_\R M_0$. This suggests that $\gamma_0^{on}$ measures the degree of reducibility of the representation.
The observation that for $kV_1$ the denominator of $\gamma_0$ is a power of $2$ can be justified
using Equation~\eqref{eq:OnShellHilbGam0}, as $s_\delta$ is a product of sums of pairs of variables.
Among all representations with fixed $\dim_\R M_0$ the irreducible representation has by far the most
intricate Hilbert series.

We also note that, with the exception of the non-$1$-large representations $V_2$ and $2V_1$,
which correspond to graded regularly symplectomorphic symplectic quotients (see
Remark~\ref{rem:Non1Large}), the $\gamma_0^{\on}(V)$ are distinct for each $V$ with
$\dim_{\R} M_0 \leq 38$. This in particular implies that there are no graded regular
symplectomorphisms among these cases. Whether this is the case in arbitrary dimension,
and in particular whether $\gamma_0^{\on}(V)$ determines $V$ for $1$-large $V$,
will be considered in a future work.

\appendix

\section{Hilbert Series in Low Dimensions}
\label{app:HSerData}

For the benefit of those readers interested in explicit descriptions of the algebra $\R[M_0]$,
we present in Table~\ref{tab:HilbM0LowDim} the Hilbert series $\Hilb_V^{\on}(t)$
of $\R[M_0]$ for symplectic quotients corresponding to $\SU_2$-modules of dimension at most $10$.
The first three cases $V_1$, $V_2$, and $2V_1$ are determined using the identification of $M_0$ with
orbifolds, see \cite[Section 5]{HerbigSchwarzSeaton} or \cite{ArmsGotayJennings,GotayBos}. The
other cases were computed using the algorithm based on Corollary~\ref{cor:OnShellHilb}
described in Section~\ref{sec:SympQuots}, which has been implemented on
\emph{Mathematica} \cite{Mathematica} and is available from the authors upon request.
The time to compute $\Hilb_V^{\on}(t)$ a PC varies widely even for representations
of the same dimension; the most time-consuming example presented here was $3V_1\oplus V_3$,
which took $107$ minutes, while $V_9$ was computed in 29 seconds and $5V_1$ was computed in
less than $2$ seconds.

When space prohibits the ordinary expression of a rational function, we use the following
abbreviated notation for the numerator. The expression $\{a_0, a_1, a_2, \ldots a_k; \; n \}$
indicates the palindromic polynomial of total degree $n$ that begins
\[
    a_0 + a_1 t + a_2 t^2 + \cdots + a_k t^k + \cdots .
\]
Note that this could either refer to
\[
    a_0 + a_1 t + a_2 t^2 + \cdots + a_k t^k + a_{k-1} t^{k+1} + a_{k-2} t^{k+2}
        + \cdots + a_0 t^n
\]
or
\[
    a_0 + a_1 t + a_2 t^2 + \cdots + a_k t^k + a_k t^{k+1} + a_{k-1} t^{k+2}
        + \cdots + a_0 t^n;
\]
this ambiguity is removed by indicating the value of $n$.

\begin{longtable}[H]{|c|c|c|c|c|}
 \hline
 $\bs{V}$   &   $\bs{\dim_{\R} M_0}$  &   $\bs{\Hilb_V^{\on}(t)}$    &   $\bs{\gamma_0^{on}}$  &   $\bs{\gamma_2^{on}}$
 \endhead
 \hline \multicolumn{5}{|c|}{$\bs{\dim_{\C} V = 2}$}
 \\
 \hline
  $V_1$         &       $0$         &
  $1$           &       $1$         &       $0$
 \\
 \hline \multicolumn{5}{|c|}{$\bs{\dim_{\C} V = 3}$}
 \\
 \hline
  $V_2$         &       $2$         &
  $\frac{1 + t^2}{(1 - t^2)^2}$ &
  $\frac{1}{2}$ &
  $\frac{1}{8}$
 \\
 \hline \multicolumn{5}{|c|}{$\bs{\dim_{\C} V = 4}$}
 \\
 \hline
  $2V_1$        &       $2$         &
  $\frac{1 + t^2}{(1 - t^2)^2}$ &
  $\frac{1}{2}$ &
  $\frac{1}{8}$
 \\
 \hline
  $V_3$         &       $2$         &
  $\frac{1+t^4}{(1-t^2)(1-t^4)}$ &
  $\frac{1}{4}$ &
  $\frac{5}{16}$
 \\
 \hline \multicolumn{5}{|c|}{$\bs{\dim_{\C} V = 5}$}
 \\
 \hline
 $V_4$         &       $4$         &
  $\frac{1+t^2+2 t^3+t^4+t^6}{(1-t^2)^2 (1-t^3)^2}$ &
  $\frac{1}{6}$ &
  $\frac{1}{8}$
 \\
 \hline
 $V_1\oplus V_2$         &       $4$         &
  $\frac{1+2 t^2+2 t^3+2 t^4+t^6}{(1-t^2)^2 (1-t^3)^2}$ &
  $\frac{2}{9}$ &
  $\frac{11}{108}$
 \\
 \hline \multicolumn{5}{|c|}{$\bs{\dim_{\C} V = 6}$}
 \\
 \hline
 $3V_1$         &       $6$         &
  $\frac{1+9 t^2+9 t^4+t^6}{(1-t^2)^6}$ &
  $\frac{5}{16}$ &
  $\frac{3}{64}$
 \\
 \hline
 $2V_2$         &       $6$         &
  $\frac{1+4 t^2+4 t^4+t^6}{ (1-t^2)^6}$ &
  $\frac{5}{32}$ &
  $\frac{11}{128}$
 \\
 \hline
 $V_1\oplus V_3$         &       $6$         &
  $\frac{1 + 12 t^4 + 13 t^6 + 13 t^8 + 12 t^{10} + t^{14} }
    { (1-t^2)^2 (1-t^4)^4 }$ &
  $\frac{13}{256}$ &
  $\frac{27}{1024}$
 \\
 \hline
 $V_5$         &       $6$         &
  $\frac{ \{1, 0, 0, 0, 2, 0, 1, 0, 14, 0, 13, 0, 29, 0, 16; \; 28 \} }
    {(1-t^2) (1-t^4) (1-t^6)^2 (1-t^8)^2}$ &
  $\frac{17}{2304}$ &
  $\frac{113}{27648}$
 \\
 \hline \multicolumn{5}{|c|}{$\bs{\dim_{\C} V = 7}$}
 \\
 \hline
 $2V_1\oplus V_2$         &       $8$         &
  $\frac{  \{1, 0, 5, 14, 13, 22, 34; \; 12\}  }
    {(1-t^2)^4 (1-t^3)^4}$ &
  $\frac{1}{9}$ &
  $\frac{13}{324}$
 \\
 \hline
 $V_2\oplus V_3$         &       $8$         &
  $\frac{ \{1, 0, 2, 2, 9, 15, 24, 36, 44, 57, 64; \; 21\} }
    {(1-t^2)^2 (1-t^3)^2 (1-t^4) (1-t^5)^3}$ &
  $\frac{127}{4500}$ &
  $\frac{721}{54000}$
 \\
 \hline
 $V_1\oplus V_4$         &       $8$         &
  $\frac{ \{{1,0,2,0,3,6,16,12,19,18; \; 18}\}}
    {(1-t^2)^2 (1-t^3)^4 (1-t^5)^2}$ &
  $\frac{34}{2025}$ &
  $\frac{221}{24300}$
 \\
 \hline
 $V_6$         &       $8$         &
  $\frac{\scriptsize\begin{array}{c}\{1, 0, 2, 0, 5, 0, 17, 8, 38, 25, 71, 64, 120, 125, 177, \\
     195, 240, 252, 299, 295, 316; \; 41 \}\end{array}}
  {(1-t^2)(1-t^4)^3 (1-t^6) (1-t^9) (1 - t^{10})^2}$ &
  $\frac{5}{768}$ &
  $\frac{17}{5120}$
 \\
 \hline \multicolumn{5}{|c|}{$\bs{\dim_{\C} V = 8}$}
 \\
 \hline
 $4V_1$         &       $10$         &
  $\frac{1+18 t^2+65 t^4+65 t^6+18 t^8+t^{10}}{(1-t^2)^{10}}$ &
  $\frac{21}{128}$ &
  $\frac{21}{512}$
 \\
 \hline
 $V_1\oplus 2V_2$         &       $10$         &
  $\frac{  \{1, 0, 5, 8, 24, 28, 48, 44; \; 14\} }
    {(1-t^2)^6 (1-t^3)^4}$ &
  $\frac{17}{324}$ &
  $\frac{167}{7776}$
 \\
 \hline
 $2V_1\oplus V_3$         &       $10$         &
  $\frac{  \{1, 0, 2, 0, 59, 0, 89, 0, 340, 0, 240; 20 \} }
    {(1-t^2)^5 (1-t^4)^5}$ &
  $\frac{611}{16384}$ &
  $\frac{1107}{65536}$
 \\
 \hline
 $2V_3$         &       $10$         &
  $\frac{  \{ 1, 0, 1, 0, 21, 0, 35, 0, 130, 0, 100; \;20\}}
    {(1-t^2)^5 (1-t^4)^5}$ &
  $\frac{119}{8192}$ &
  $\frac{215}{32768}$
 \\
 \hline
 $V_2\oplus V_4$         &       $10$         &
  $\frac{  \{ 1, 0, 0, 6, 13, 8, 19, 28; \; 14\} }
    { (1-t^2)^6 (1-t^3)^4}$ &
  $\frac{61}{2592}$ &
  $\frac{353}{31104}$
 \\
 \hline
 $V_1\oplus V_5$         &       $10$         &
  $\frac{\scriptsize\begin{array}{c}\{ 1, 0, 1, 0, 11, 0, 68, 0, 286, 0, 746, 0, 1820, 0, 3451, \\
    0, 5733, 0, 8042, 0, 9993, 0, 10532; \; 44 \} \end{array}}
    { (1-t^2) (1-t^4)^3 (1-t^6)^4 (1-t^8)^2 }$ &
  $\frac{5903}{884736}$ &
  $\frac{36461}{10616832}$
 \\
 \hline
 $V_7$         &       $10$         &
  $\frac{\scriptsize\begin{array}{c}\{ 1,0,1,0,4,0,17,0,100,0,301,0,967, 0, 2333, 0, 5291, 0, 10464, \\
    0, 19436, 0, 32516, 0, 51410, 0, 74928, 0, 103252, 0, 132793, \\
    0, 162204, 0, 185681, 0, 202349, 0, 207442; \; 76 \} \end{array} }
    {(1-t^4) (1-t^6)^2 (1-t^8)^2 (1-t^{10})^3 (1-t^{12})^2}$ &
  $\frac{1087769}{663552000}$ &
  $\frac{6309547}{7962624000}$
 \\
 \hline \multicolumn{5}{|c|}{$\bs{\dim_{\C} V = 9}$}
 \\
 \hline
 $3V_1\oplus V_2$         &       $12$         &
  $\frac{ \{ 1, 0, 12, 34, 62, 158, 297, 366, 486, 580; \; 18 \} }{ (1 - t^2)^6 (1 - t^3)^6 }$ &
  $\frac{853}{11664}$ &
  $\frac{1211}{46656}$
 \\
 \hline
 $3V_2$         &           $12$         &
  $\frac{ \{1, 0, 9, 14, 30, 48, 44; \; 12\} }
    { (1 - t^2)^{12} }$ &
  $\frac{31}{512}$ &
  $\frac{51}{2048}$
 \\
 \hline
 $V_1\oplus V_2\oplus V_3$         &    $12$         &
  $\frac{ \scriptsize\begin{array}{c} \{1, 0, 3, 15, 43, 106, 247, 510, 959, 1662, 2674, 3983, \\
        5578, 7281, 8962, 10378, 11329, 11644; \; 34 \}  \end{array} }
    { (1 - t^2)^2 (1 - t^3)^3 (1 - t^4)^3 (1 - t^5)^3 (1 - t^6) }$ &
  $\frac{6617}{288000}$ &
  $\frac{20803}{2073600 }$
 \\
 \hline
 $2V_1\oplus V_4$         &       $12$         &
  $\frac{ \scriptsize\begin{array}{c} \{1, 0, 7, 0, 27, 64, 177, 308, 619, 1036, 1692, 2618, 3715, \\
        4950, 6311, 7664, 8632, 9348, 9614; \; 36\}  \end{array} }
        { (1 - t^2)^2 (1 - t^3)^4 (1 - t^5)^4 (1 - t^6)^2 }$ &
  $\frac{6497}{455625}$ &
  $\frac{12773}{1822500}$
 \\
 \hline
 $V_3\oplus V_4$         &        $12$         &
  $\frac{\scriptsize\begin{array}{c}\{  1, 0, 3, 2, 16, 31, 96, 196, 419, 739, 1285, 2018, 3106, 4453, 6190, \\
    8114, 10251, 12290, 14195, 15628, 16604, 16888; \; 42 \}\end{array}}
    {(1-t^2) (1-t^3)^4 (1-t^4) (1-t^5)^3 (1-t^7)^3}$ &
  $\frac{104081}{13891500}$ &
  $\frac{112691}{33339600}$
 \\
 \hline
 $V_2\oplus V_5$         &        $12$         &
  $\frac{ \scriptsize\begin{array}{c}\{ 1, 0, 4, 4, 18, 33, 103, 227, 527, 1088, 2201, 4159, 7564, 13162, \\
    22088, 35778, 56103, 85378, 126257, 181801, 255208, 349731, \\
    468381, 613621, 787245, 989611, 1220152, 1476055, 1753528,  \\
    2046090, 2346792, 2646000, 2934505, 3201233, 3436754, \\
    3630791, 3776029, 3865470, 3895992; \; 76  \} \end{array} }
    {  (1 - t^5)^3 (1 - t^6)^2 (1 - t^7)^3 (1 - t^8)^2 (1 - t^{12})^2  }$ &
  $\frac{4785211}{889056000}$ &
  $\frac{27426803}{10668672000}$
 \\
 \hline
 $V_1\oplus V_6$         &        $12$         &
  $\frac{\scriptsize\begin{array}{c} \{1, 0, 4, 0, 13, 14, 63, 116, 295, 564, 1161, 2020, 3683, 5916, 9678, \\
    14566, 21837, 30762, 42930, 56848, 74413, 93114, 114990, 136452, \\
    159818, 180478, 201079, 216702, 230366, 237126, 241006; \; 60 \} \end{array} }
        {(1-t^4)^2 (1-t^5)^4 (1-t^6)^3 (1-t^7)^2 (1-t^{12}) }$ &
  $\frac{244439}{79380000}$ &
  $\frac{2897051}{1905120000}$
 \\
 \hline
 $V_8$         &            $12$         &
  $\frac{\scriptsize\begin{array}{c} \{ 1, 0, 3, 3, 9, 18, 43, 84, 179,326,604,1015,1706,2655,4082, \\
    5914, 8367, 11262, 14751, 18428, 22410, 26071, 29490, \\
    32017, 33793, 34264; \; 50 \} \end{array} }
        {(1 - t^3) (1 - t^4)^3 (1 - t^5)^3 (1 - t^6)^3 (1 - t^7)^2}$ &
  $\frac{32909}{18144000}$ &
  $\frac{36809}{43545600}$
 \\
 \hline \multicolumn{5}{|c|}{$\bs{\dim_{\C} V = 10}$}
 \\
 \hline
 $5V_1$         &       $14$         &
  $\frac{ \{ {1, 0, 31, 0, 231, 0, 595, 0; \; 14} \} }{ (1 - t^2)^{14} }$ &
  $\frac{429}{4096}$ &
  $\frac{495}{16384}$
 \\
 \hline
 $2V_1\oplus 2V_2$         &       $14$         &
  $\frac{  \{ 1, 0, 10, 32, 98, 220, 488, 860, 1366, 1836, 2253, 2376; \; 22 \}  }
    { (1 - t^2)^6 (1 - t^3)^8 }$ &
  $\frac{29}{729}$ &
  $\frac{5}{324}$
 \\
 \hline
 $3V_1\oplus V_3$         &       $14$         &
  $\frac{  \{ 1, 0, 9, 0, 179, 0, 762, 0, 3375, 0, 6834, 0, 12999, 0, 13524; \; 28 \} }
    { (1 - t^2)^7 (1 - t^4)^7 }$ &
  $\frac{30921}{1048576}$ &
  $\frac{52659}{4194304}$
 \\
 \hline
 $2V_2\oplus V_3$         &       $14$         &
  $\frac{  \scriptsize\begin{array}{c}\{1, 0, 7, 9, 51, 134, 366, 784, 1593, 2947, 5199, 8400, 12830, \\
    18152, 24504, 31023, 37472, 42613, 46145, 47252; \; 38\}  \end{array}}
        { (1 - t^2)^4 (1 - t^3)^3 (1 - t^4) (1 - t^5)^5 (1 - t^6) }$ &
  $\frac{15991}{1012500}$ &
  $\frac{81041}{12150000}$
 \\
 \hline
 $V_1\oplus 2V_3$         &       $14$         &
  $\frac{  \{1, 0, 0, 0, 93, 0, 286, 0, 1569, 0, 2758, 0, 5901, 0, 5530; \; 28 \}  }
    { (1 - t^2)^7 (1 - t^4)^7 }$ &
  $\frac{13373}{1048576}$ &
  $\frac{23199}{4194304}$
 \\
 \hline
 $V_1\oplus V_2\oplus V_4$         &       $14$         &
  $\frac{  \{1, 0, 1, 10, 29, 68, 156, 268, 446, 724, 1015, 1214, 1406, 1500; \; 26 \} }
    { (1 - t^2)^6 (1 - t^3)^6 (1 - t^5)^2 }$ &
  $\frac{761}{72900}$ &
  $\frac{2789}{583200}$
 \\
 \hline
 $2V_4$         &       $14$         &
  $\frac{  \{1, 0, 2, 14, 17, 24, 92, 154, 161, 234, 306; \; 20 \} }
    { (1 - t^2)^8 (1 - t^3)^6 }$ &
  $\frac{71}{7776}$ &
  $\frac{125}{31104}$
 \\
 \hline
 $2V_1\oplus V_5$         &       $14$         &
  $\frac{ \scriptsize\begin{array}{c} \{1, 0, 6, 0, 46, 0, 454, 0, 2849, 0, 12140, 0, 43131, 0, 127076, \\
    0, 315389, 0, 673304, 0, 1260139, 0, 2076447, 0, 3042040, \\
    0, 3982739, 0, 4675695, 0, 4928416; \; 60\} \end{array} }
        {  (1 - t^2) (1 - t^4)^5 (1 - t^6)^6 (1 - t^8)^2 }$ &
  $\frac{1167229}{191102976}$ &
  $\frac{127411}{42467328}$
 \\
 \hline
 $V_3\oplus V_5$         &       $14$         &
  $\frac{  \scriptsize\begin{array}{c} \{ 1, 0, 0, 0, 40, 0, 235, 0, 1536, 0, 6245, 0, 22073, 0, 62288, 0, \\
    153198, 0, 322982, 0, 604168, 0, 1000931, 0, 1491320, \\
    0, 1998930, 0, 2427434, 0, 2672013, 0; \; 62 \} \end{array} }
    { (1 - t^2)^2 (1 - t^4)^4 (1 - t^6)^4 (1 - t^8)^4 }$ &
  $\frac{1793899}{452984832}$ &
  $\frac{9710395}{5435817984}$
 \\
 \hline
 $V_2\oplus V_6$         &       $14$         &
  $\frac{  \scriptsize\begin{array}{c} \{ 1, 0, 3, 0, 35, 34, 195, 318, 899, 1580, 3412, 5788, 10695, \\
    17170, 28357, 43056, 65617, 94006, 134421, 182888, \\
    246941, 320684, 411834, 511564, 628197, 748348, \\
    880927, 1009232, 1141267, 1258738, 1370691, \\
    1456394, 1528351, 1566260, 1584434; \; 68\} \end{array}  }
    { (1 - t^2)^3 (1 - t^4)^5 (1 - t^9)^4 (1 - t^{10})^2 }$ &
  $\frac{4463}{829440}$ &
  $\frac{25219}{9953280}$
 \\
 \hline
 $V_1\oplus V_7$         &       $14$         &
  $\frac{  \scriptsize\begin{array}{c} \{ 1, 0, 2, 0, 16, 0, 120, 0, 949, 0, 4484, 0, 19061, 0, 66638, 0, \\
    206241, 0, 563855, 0, 1399730, 0, 3161375, 0, 6596301, 0, \\
    12755465, 0, 23052381, 0, 39054709, 0, 62358940, 0, \\
    94039452, 0, 134429968, 0, 182432050, 0, 235602923, \\
    0, 289833599, 0, 340185063, 0, 381119164, 0, \\
    407976216, 0, 417274692; \; 100 \} \end{array} }
        { (1 - t^4) (1 - t^6)^4 (1 - t^8)^4 (1 - t^{10})^3 (1 - t^{12})^2 }$ &
  $\frac{2423496049}{1528823808000}$ &
  $\frac{184571731}{244611809280}$
 \\
 \hline
 $V_9$         &       $14$         &
  $\frac{  \scriptsize\begin{array}{c} \{ 1, 0, 1, 0, 10, 0, 42, 0, 334, 0, 1566, 0, 6958, 0, 25277, 0, 83391, \\
    0, 244771, 0, 662241, 0, 1652020, 0, 3858520, 0, 8466785, 0, \\
    17599687, 0, 34772336, 0, 65630156, 0, 118662007, 0, \\
    206217754, 0, 345216158, 0, 558013931, 0, 872420918, \\
    0, 1321591127, 0, 1942389147, 0, 2773434697, 0, \\
    3851193503, 0, 5206038355, 0, 6856598397, 0, \\
    8805135769, 0, 11032270863, 0, 13494344349, 0, \\
    16121552188, 0, 18820073941, 0, 21475826372, 0, \\
    23962298579, 0, 26149487910, 0, 27915236756, \\
    0, 29155789630, 0, 29795890397, 0; \; 154 \} \end{array} }
        { (1 - t^8)^2 (1 - t^{10})^3 (1 - t^{12})^4 (1 - t^{14})^3 (1 - t^{16})^2  }$ &
  $\frac{62728171711}{116530348032000}$ &
  $\frac{68283510691}{279672835276800}$
 \\
 \hline
\caption{Hilbert series of $\R[M_0]$ for symplectic quotients $M_0$ corresponding to low-dimensional $\SU_2$-modules $V$.}
\label{tab:HilbM0LowDim}
\end{longtable}


\section{Visualization of $\gamma_0^{\on}(V)$}
\label{app:Gam0Plot}

Here, we give three graphs of the values of $\gamma_0^{\on}(V)$ to illustrate its dependence
on the weights of $V$, the number of irreducible subrepresentations of $V$, and dimension of $M_0$.
In Figure~\ref{fig:2DGam0Dim}, the horizontal axis is the dimension of $M_0$ and the vertical
axis is $-\log \gamma_0^{\on}(V)$.
In Figure~\ref{fig:3DGam0SumDim}, the horizontal axes are the dimension of $M_0$
and the sum of the positive weights in $V$, while the vertical axis is $-\log\gamma_0^{\on}(V)$.
In Figure~\ref{fig:3DGam0IrrepsDim}, the horizontal axes are the dimension of $M_0$
and the number of irreducible subrepresentations of $V$, and the vertical axis is again
$-\log\gamma_0^{\on}(V)$.
All three graphs plot the value of $\gamma_0^{\on}(V)$ for every $V$ such that
$V^{\SU_2} = \{0\}$ and $\dim M_0 \leq 38$.

These values were computed with \emph{Mathematica} \cite{Mathematica} using the expression in
Theorem~\ref{thrm:OnShellHilbGam0}.
The Schur polynomials were computed by expressing them in terms of elementary symmetric
polynomials using the Jacobi-Trudi identities \cite[Theorem 4.5.1]{SaganSymGrp}, which
is much faster than via the definition as a quotient of alternants.

\begin{figure}[h!]
\includegraphics[scale=1.2]{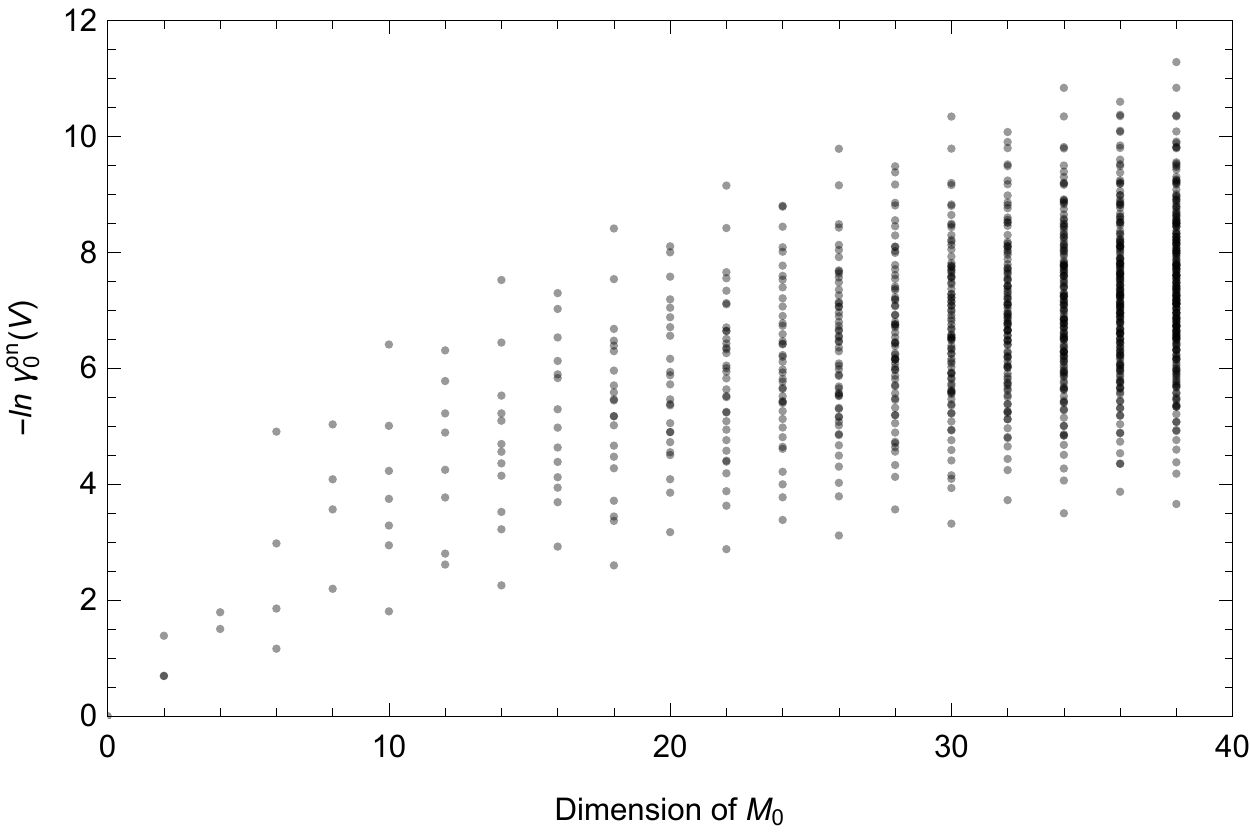}
\caption{A plot of the values of $-\ln\gamma_0^{\on}(V)$ vs. $\dim_{\R} M_0$.
Includes all $V$ with $V^{\SU_2}=\{0\}$ and $\dim_{\R} M_0 \leq 38$.
For fixed $\dim_{\R} M_0$, the largest or second largest $-\gamma_0^{on}$ comes from the irreducible representation. The smallest value of $-\gamma_0^{on}$ comes from the representation with $k V_1\oplus V_2$ if $\dim_{\R} M_0$ is divisible by four and $k V_1$ otherwise.}
\label{fig:2DGam0Dim}
\end{figure}

\begin{figure}[h!]
\includegraphics[scale=1.2]{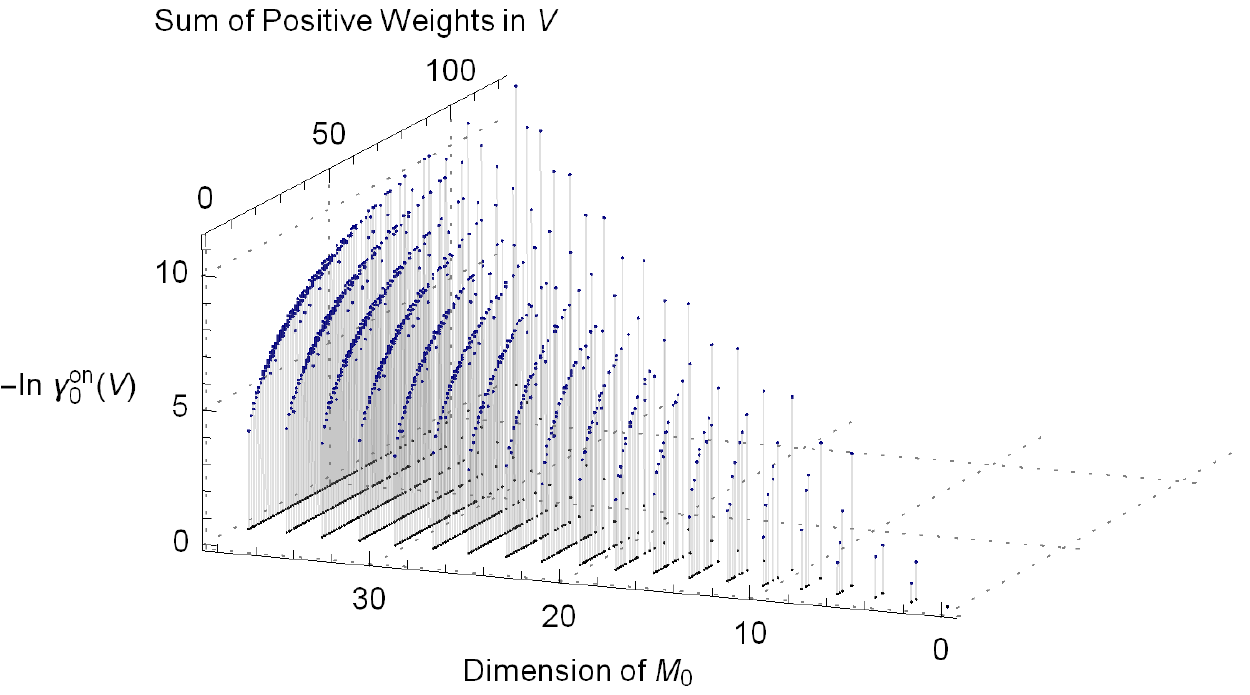}
\caption{A plot of the values of $-\ln\gamma_0^{\on}(V)$
vs. the sum of the positive weights in $V$ (i.e. elements of $\Lambda_V$) and $\dim_{\R} M_0$.
Includes all $V$ with $V^{\SU_2}=\{0\}$ and $\dim_{\R} M_0 \leq 38$.}
\label{fig:3DGam0SumDim}
\end{figure}

\begin{figure}[h!]
\includegraphics[scale=1.2]{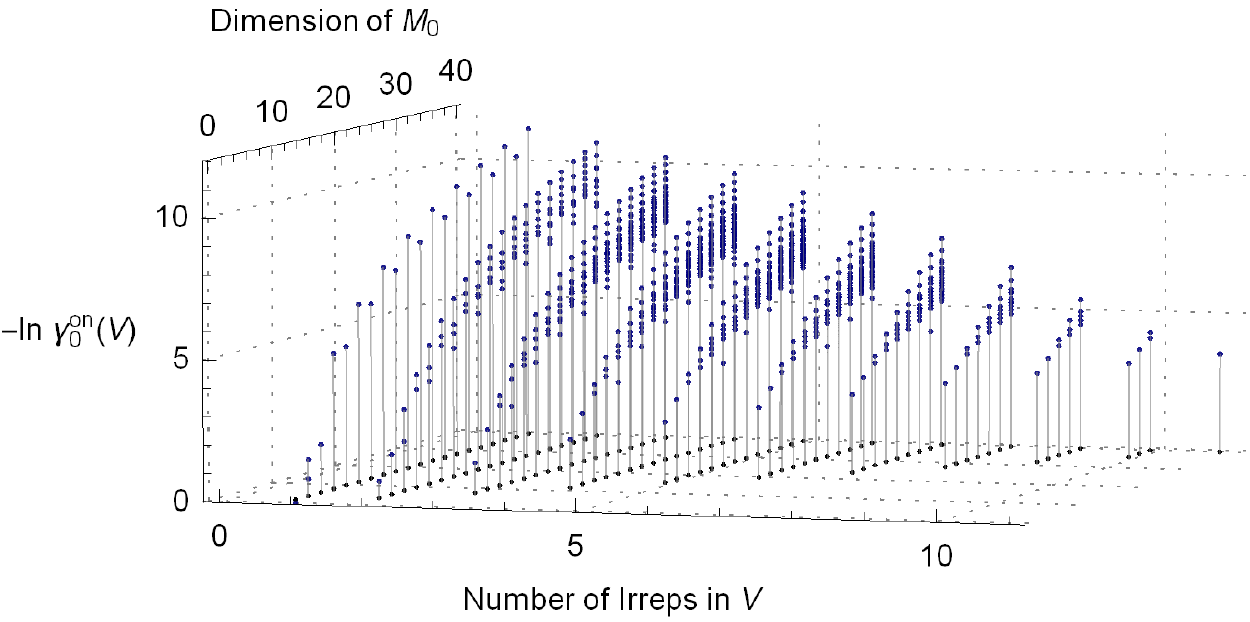}
\caption{A plot of the values of $-\ln\gamma_0^{\on}(V)$
vs. the number of irreducible summands in $V$ (i.e. $r$) and $\dim_{\R} M_0$.
Includes all $V$ with $V^{\SU_2}=\{0\}$ and $\dim_{\R} M_0 \leq 38$.}
\label{fig:3DGam0IrrepsDim}
\end{figure}


\bibliographystyle{amsplain}
\bibliography{HHS-SU2-onshell}

\end{document}